\documentclass[a4paper,11pt,reqno]{amsart}

\usepackage{amsmath}
\usepackage{amssymb}
\usepackage{paralist}
\usepackage{graphicx}
\usepackage{epstopdf}
\usepackage{subfig}
\usepackage{tikz}
\usepackage{pgfplots}
\usepackage{xcolor}
\usepackage{algorithm}
\usepackage{algpseudocode}
\usepackage{comment}
\usepackage[colorlinks=true,linkcolor=blue,urlcolor=blue,citecolor=red]{hyperref}

\pgfplotsset{compat=1.18}

\newtheorem{theorem}{Theorem}[section]

\newtheorem{lemma}[theorem]{Lemma}
\newtheorem{proposition}[theorem]{Proposition}

\newtheorem{assumption}[theorem]{Assumption}

\theoremstyle{definition}
\newtheorem{definition}[theorem]{Definition}
\newtheorem{example}[theorem]{Example}
\newtheorem{remark}[theorem]{Remark}

\newcommand{\dx}{\,\mathrm{d}x}

\title[Numerics for Partially Segregated Elliptic Systems]{}
\author{}

\begin{document}

\begin{center}
{\LARGE\textbf{Numerical Algorithms for Partially Segregated Elliptic Systems}}
\end{center}

\vspace{2em}

\begin{center}
\begin{minipage}[t]{\textwidth}
\centering
\begin{minipage}[t]{0.45\textwidth}
\centering
\textbf{Farid Bozorgnia}\\
\small Department of Mathematics\\
\small New Uzbekistan University\\
\small Tashkent, Uzbekistan\\
\small \texttt{drfaridba@gmail.com}
\end{minipage}%
\hfill%
\begin{minipage}[t]{0.45\textwidth}
\centering
\textbf{Avetik Arakelyan}\\
\small Institute of Mathematics, NAS of Armenia\\
\small Yerevan State University\\
\small Yerevan, Armenia\\
\small \texttt{avetik.arakelyan@ysu.am}
\end{minipage}

\vspace{1.5em}

\begin{minipage}[t]{0.45\textwidth}
\centering
\textbf{Vyacheslav Kungurtsev}\\
\small Department of Computer Science\\
\small Czech Technical University in Prague, Faculty of Electrical Engineering\\
\small Prague, Czech Republic\\
\small \texttt{vyacheslav.kungurtsev@fel.cvut.cz}
\end{minipage}%
\hfill%
\begin{minipage}[t]{0.45\textwidth}
\centering
\textbf{Jan Valdman}\\
\small Department of Applied Mathematics \\
\small Czech Technical University in Prague, Faculty of Information Technology \\
\small Prague, Czech Republic\\
\small \texttt{valdmjan@cvut.cz}
\end{minipage}
\end{minipage}
\end{center}
\begin{abstract}
We develop numerical methods for elliptic systems governed by a ternary
partial-segregation constraint, in which three nonnegative components are
required to have a vanishing pointwise product throughout the domain. This
condition forces at least one component to vanish at every spatial location
while still permitting coexistence of the other two components, and therefore
produces a highly nonconvex admissible set. In particular, the limiting
constrained Dirichlet minimization problem need not have a unique minimizer, as
we show by an explicit one-dimensional example with two distinct global
minimizers. We propose two complementary computational frameworks. The first is
a strong-competition penalty method solved by a relaxed Picard iteration and a
sequential semi-implicit variant, combined with continuation in the penalty
parameter. For this formulation, we establish existence of penalized minimizers
and weak solutions for every fixed penalty parameter, monotone bracketing
properties for the fixed-point iteration, geometric convergence under an
explicit contraction condition, and interior exponential estimates in regions of
uniformly positive competition. The second framework is a discrete
projected-gradient method, together with a safeguarded inertial variant, based
on an explicit pointwise projection onto the ternary partial-segregation set.
Numerical experiments for several discrete boundary configurations show that,
when started on a common solution branch, the two methods produce the same
interfaces, and that apparent differences between them reflect the selection of
different branches by their default initializations.
\end{abstract}


\maketitle
\bigskip

\tableofcontents

\section{Introduction}\label{sec:Intro}
In recent decades, there has been considerable interest in spatial segregation phenomena arising in reaction--diffusion systems. In particular, spatially inhomogeneous solutions of Lotka--Volterra-type competition models involving two or more interacting species have been extensively investigated; see, for example, \cite{VZ2014,MR2146353,WZ2010,WeiWeth2008,MR2151234,MR2417905}. In the strong-competition limit, these systems lead to pairwise segregated configurations satisfying $u_i u_j=0$ for $i\neq j$ and to multi-phase free-boundary problems describing the separation of competing components. Related variational and free-boundary models occur in shape optimization, quadrature-domain problems, and optimal partition theory; see \cite{Avet-Henrik,bucur-multi,Aram-Caff,farid-optimal-partition} and the references therein.

In this paper, we investigate numerical schemes for a class of elliptic systems describing competitive interactions among components $u_1,\ldots,u_m$, where each $u_i\colon\Omega\to[0,\infty)$ represents the density or concentration of the $i$-th component on a bounded domain $\Omega\subset\mathbb{R}^d$. The components satisfy the partial-segregation constraint
\[
\prod_{i=1}^m u_i(x)=0,
\]
which requires at least one component to vanish at each spatial point but allows as many as $m-1$ components to coexist. Consequently, the domain is stratified according to the local vanishing and coexistence pattern rather than partitioned into mutually exclusive single-component phases. Although numerical methods for pairwise segregation are well developed, a systematic numerical treatment of the partial-segregation constraint, including both penalized and projection-based algorithms, appears to be largely absent. This gap is a principal motivation for the present work.

Ternary interaction systems arise, for example, in variational models of multicomponent mixtures with three-body competition. Their singular limits and the associated partially segregated harmonic maps have recently been studied from the analytical viewpoint in \cite{soave2024partial1,soave2024partial2,bozorgnia2018singularly,bozorgnia2025gamma}. The emphasis of the present paper is computational, but the numerical schemes are accompanied by fixed-parameter existence, stability, and convergence analysis.

\medskip

Optimization over a nonconvex segregation set is difficult for two distinct reasons. First, the feasible set contains many competing active-set patterns and is
genuinely nonconvex; in fact, even the limiting Dirichlet minimization
problem can possess multiple global minimizers, as shown by the elementary
one-dimensional Example~\ref{ex:nonuniqueness} below.  Second, the product constraint is not uniformly regular: at points where at least two components vanish, its derivative is degenerate, while the active vanishing component may switch across the domain. Thus standard smooth nonlinear-programming theory and classical constraint qualifications cannot be applied globally without additional analysis; see \cite{ito2003lagrange,Wachsmuth2015}. A complete stationarity theory for the limiting function-space constraint is beyond the scope of this paper.

Traditional numerical methods for segregation problems also face stiffness as the competition parameter becomes small and nonsmooth changes in the active phase pattern. We now state the scope and contributions of the paper. Throughout, $\Omega\subset\mathbb{R}^d$ is a bounded Lipschitz domain. From Section~\ref{sec: Penal_algorithms} onward, we restrict the model to $m=3$ components because both the ternary interaction and the explicit projection formula are three-component constructions. No restriction $d\le3$ is required for the variational or fixed-$\varepsilon$ arguments below; the relevant nonlinear terms are controlled by componentwise truncation and the maximum principle. Any stronger dimensional or boundary-regularity hypothesis needed for a later local regularity result is stated at that point. Our contributions are as follows.
\begin{enumerate}
\item We formulate two complementary algorithms for the partial-segregation problem: a penalization method with a relaxed Picard iteration, a sequential semi-implicit variant, and $\varepsilon$-continuation (Section~\ref{subsec: Pickard_iter}), and a discrete projected-gradient method, with a safeguarded inertial variant, based on an explicit pointwise projection onto the ternary partial-segregation set (Section~\ref{sec:pg}).
\item For every fixed $\varepsilon>0$, global minimizers of the penalized energy exist and satisfy the penalized Euler--Lagrange system. We also prove existence through the fixed-point formulation (Theorem~\ref{thm:existence_schauder}), establish monotone even and odd subsequences that bracket every fixed point of the undamped Picard map (Theorem~\ref{thm:sandwich}), and show that, when $\varepsilon\lambda_\Omega>4M^4$, the Picard map is a contraction and the corresponding relaxed iteration converges geometrically to the unique fixed point (Theorem~\ref{thm:contraction}).
\item For global penalized minimizers, Proposition~\ref{prop:penalty_conv} proves strong subsequential convergence in $(H^1(\Omega))^3$ to a global minimizer of the limiting constrained problem as $\varepsilon\downarrow0$. For the iterative solvers, we establish interior estimates (Lemmas~\ref{lem:A}  and Lemmas~\ref{lem:1.8}--\ref{lem:1.9}) showing exponential damping on compact subsets where the relevant competition coefficient is bounded below; Theorem~\ref{thm:interior} makes this precise. A convergence estimate for the complete fixed-point algorithm that is uniform as $\varepsilon\downarrow0$ is not claimed; see Remark~\ref{rem:no_global}.
\item We test both algorithms on nine discrete boundary configurations adapted from \cite{Bidentical} and compare their threshold-interface geometries qualitatively; see Section~\ref{sec: simulations}.
\end{enumerate}
A full first-order stationarity theory for arbitrary local minimizers or algorithmic accumulation points of the limiting constrained problem, as well as quantitative convergence rates in $\varepsilon$, is outside the scope of this work.

\medskip

The remainder of the paper is organized as follows. Section~\ref{sec: Previous Work} introduces the mathematical framework and reviews relevant results from the segregation literature. Section~\ref{sec: Penal_algorithms} and Section~\ref{sec:pg} present the two numerical methods and their analytical properties. Section~\ref{sec: simulations} contains the numerical experiments.

\section{Mathematical Background and Previous Work}\label{sec: Previous Work}

Consider the Dirichlet energy
\begin{equation}\label{eq:constrained_energy_intro}
E(U)=\int_{\Omega}\sum_{i=1}^{m}|\nabla u_i(x)|^2\dx
\end{equation}
on the boundary-constrained partial-segregation set
\begin{equation}\label{eq:constraint_set_intro}
S:=\left\{U=(u_1,\ldots,u_m)\in(H^1(\Omega))^m:
\begin{array}{l}
 u_i\ge0\ \text{a.e. in }\Omega,\quad \operatorname{Tr}u_i=\phi_i,\\[1mm]
 \displaystyle\prod_{i=1}^m u_i=0\ \text{a.e. in }\Omega
\end{array}
\right\}.
\end{equation}
In the rest of the paper, $\operatorname{Tr}$  denotes the Sobolev trace of $u\in H^{1}(\Omega)$ on $\partial \Omega.$

\begin{assumption}\label{ass:boundary_data}
The domain $\Omega\subset\mathbb{R}^d$, $d\ge1$, is bounded and has Lipschitz boundary. The boundary data satisfy
\[
\phi_i\in H^{1/2}(\partial\Omega)\cap L^\infty(\partial\Omega),
\qquad \phi_i\ge0,
\qquad \prod_{i=1}^m\phi_i=0\quad\text{a.e. on }\partial\Omega.
\]
Moreover, they admit a nonnegative Lipschitz extension
$\psi=(\psi_1,\ldots,\psi_m)\in\operatorname{Lip}(\overline\Omega)^m$ such that
$\operatorname{Tr}\psi_i=\phi_i$ and $\prod_{i=1}^m\psi_i=0$ in $\Omega$.
\end{assumption}

\begin{remark}[Standing truncation and compatibility observations]
\label{rem:standing}
Two points concerning Assumption~\ref{ass:boundary_data} will be used
repeatedly.

First, no restriction on the space dimension is needed either to define
the partial-segregation constraint or for the penalized minimization
arguments below. Since $H^1(\Omega)\subset L^2(\Omega)$, each component has
an a.e.-finite measurable representative, and hence
\[
\prod_{i=1}^m u_i=0
\]
is a meaningful pointwise a.e.\ condition.

Let
\[
M_\phi
:=
\max_{1\le i\le m}
\|\phi_i\|_{L^\infty(\partial\Omega)}
\]
and define the truncation
\[
\mathcal T_{M_\phi}(s)
:=
\min\{M_\phi,\max\{s,0\}\}.
\]
By the Sobolev chain rule,
\[
|\nabla \mathcal T_{M_\phi}(u_i)|
\le
|\nabla u_i|
\qquad\text{a.e. in }\Omega,
\]
and
\[
0\le \mathcal T_{M_\phi}(u_i)\le M_\phi.
\]
Since the prescribed traces satisfy
$0\le\phi_i\le M_\phi$ a.e.\ on $\partial\Omega$,
componentwise truncation preserves the trace condition and does not
increase either the Dirichlet energy or any nonnegative product penalty
of the form
\[
\int_\Omega
\left(\prod_{i=1}^m u_i\right)^2 \dx.
\]
Consequently, minimizing sequences may be restricted to the bounded class
\[
0\le u_i\le M_\phi
\qquad\text{a.e. in }\Omega.
\]
On this class all product terms are bounded, so the penalized functional
is finite in every finite dimension. In particular, the frequently used
restriction $d\le3$, which guarantees
$(u_1u_2u_3)^2\in L^1(\Omega)$ for arbitrary unrestricted
$H^1$ triples when $m=3$, is not required for the variational arguments
used here.

Second, the extension $\psi$ in
Assumption~\ref{ass:boundary_data} belongs to
$(H^1(\Omega)\cap L^\infty(\Omega))^m$, has the prescribed traces, and
satisfies
\[
\prod_{i=1}^m\psi_i=0
\qquad\text{in }\Omega.
\]
Hence $\psi\in S$ and therefore
\[
S\neq\emptyset.
\]
The existence of such an extension is a genuine compatibility assumption:
the boundary condition
\[
\prod_{i=1}^m\phi_i=0
\qquad\text{a.e. on }\partial\Omega
\]
alone does not in general provide a segregated interior extension.
This compatibility condition is satisfied by all boundary configurations
used in Section~\ref{sec: simulations}; related extension assumptions appear in
\cite{soave2024partial1}.
\end{remark}

A closely related model is pairwise segregation. Strongly competing reaction--diffusion systems are classically written in the form
\begin{equation}\label{eq:seg-system}
\begin{cases}
-\Delta u_i+\dfrac{1}{\varepsilon}u_i\displaystyle\sum_{j\ne i}u_j^2=0
& \text{in }\Omega,\\[1mm]
u_i\ge0,\qquad \operatorname{Tr}u_i=\phi_i,
& i=1,\ldots,m,
\end{cases}
\end{equation}
where $\varepsilon>0$ is a small competition parameter and the boundary traces are pairwise segregated. Under the standard compactness hypotheses, strongly competing families satisfy
\[
u_i u_j=0\quad\text{a.e. in }\Omega,
\qquad i\ne j,
\]
in the limit $\varepsilon\downarrow0$, so that at most one component is active at each point. In the variational setting, global minimizers of the corresponding penalized interaction energy converge, along subsequences, to minimizers of the Dirichlet energy over the pairwise-segregated class with the prescribed traces. This qualification is essential: an arbitrary solution branch of \eqref{eq:seg-system} need not converge to a global energy minimizer.

The asymptotic analysis of \eqref{eq:seg-system} has been developed extensively. Compactness, uniform regularity, limiting interface structure, and connections with optimal partition problems are treated, among other works, in \cite{CaffarelliLin2007,SoaveTerracini2015,Crooks2004,farid-optimal-partition}.

\begin{definition}\label{def:classes}
Let $U=(u_1,\ldots,u_m)\in(H^1(\Omega))^m$ with $u_i\ge0$ a.e. in $\Omega$. We distinguish the following pointwise admissible classes:
\begin{itemize}
\item \emph{Pairwise segregation}:
\[
u_i u_j=0\qquad\text{a.e. in }\Omega\quad\text{for all }i\ne j,
\]
denoted by $\mathcal{S}_1$.

\item \emph{Partial segregation}:
\[
\prod_{i=1}^m u_i=0\qquad\text{a.e. in }\Omega,
\]
denoted by $\mathcal{S}$.
\end{itemize}
The boundary-constrained set $S$ in \eqref{eq:constraint_set_intro} is the subset of $\mathcal{S}$ having traces $\phi_i$. No $L^\infty$ assumption is required to define either pointwise constraint; boundedness is imposed only when needed to control a penalty or a fixed-point coefficient.
\end{definition}
The main differences between the two notions are as follows.
\begin{enumerate}
\item \emph{Coexistence structure.} Partial segregation requires at least one component to vanish at each point and therefore allows up to $m-1$ components to be simultaneously positive. Pairwise segregation permits at most one positive component.

\item \emph{Pointwise geometry.} For $i=1,\ldots,m$, define
\[
C_i:=\{a=(a_1,\ldots,a_m)\in\mathbb{R}^m:
a_i\ge0,\ a_j=0\ \text{for }j\ne i\},
\]
and, for $k=1,\ldots,m$, define
\[
D_k:=\{a\in\mathbb{R}^m:
a_j\ge0\ \text{for all }j,\ a_k=0\}.
\]
Then $U\in\mathcal{S}_1$ if and only if
$U(x)\in\bigcup_{i=1}^m C_i$ for a.e. $x$, whereas
$U\in\mathcal{S}$ if and only if
$U(x)\in\bigcup_{k=1}^m D_k$ for a.e. $x$. The sets $C_i$ are one-dimensional half-lines intersecting only at the origin. The sets $D_k$ are $(m-1)$-dimensional faces of the nonnegative orthant and overlap whenever two or more coordinates vanish. Neither function-space admissible class is a finite union of convex subsets determined by one global active index, because the active face may change from point to point.

\item For $m=3$, the vector $(0.5,0.3,0)$ belongs to $\mathcal{S}$ but not to $\mathcal{S}_1$.
\end{enumerate}


From a computational viewpoint, the direct treatment of segregation
constraints presents several difficulties. A small penalty parameter creates
a stiff elliptic system, while nonconvexity makes the computed configuration
sensitive to initialization and continuation. Moreover, the limiting
constrained minimization problem need not have a unique global minimizer.
This phenomenon already occurs in one space dimension, as the following
example shows.

\begin{example}[Nonuniqueness of the limiting constrained minimizer]
\label{ex:nonuniqueness}
Let $\Omega=(0,1)$ and $m=3$. We set the boundary data for 
$U=(u_1,u_2,u_3)$ as follows
\[
U(0)=(2, 1, 0), \qquad U(1)=(0, 1, 3).
\]
Notice that at each boundary, exactly two components are strictly positive, which is the maximum possible number of non-zero entries allowed by the constraint $\prod_{i=1}^3 u_i = 0$.
Consider the constrained Dirichlet problem
\[
\min\left\{
E(U):= \int_0^1 \left( |u_1'(x)|^2+|u_2'(x)|^2+|u_3'(x)|^2 \right)\dx : U\in S
\right\}.
\]
This problem has at least two distinct global minimizers. 
Define the first configuration, $U^{(A)}$, in which $u_2$ is strictly positive everywhere
\[
U^{(A)}(x) = 
\begin{pmatrix}
(2 - 5x)^+\\[1mm]
1\\[1mm]
(5x - 2)^+
\end{pmatrix}.
\]
Define the second configuration, $U^{(B)}$, in which $u_2$ vanishes on the middle interval $[1/4, 2/3]$ as follows
\[
U^{(B)}(x) = 
\begin{pmatrix}
(2 - 3x)^+\\[1mm]
(1 - 4x)^+ + (3x - 2)^+\\[1mm]
(4x - 1)^+
\end{pmatrix}.
\]

\begin{figure}[htbp]
\centering
\includegraphics[scale=0.6]{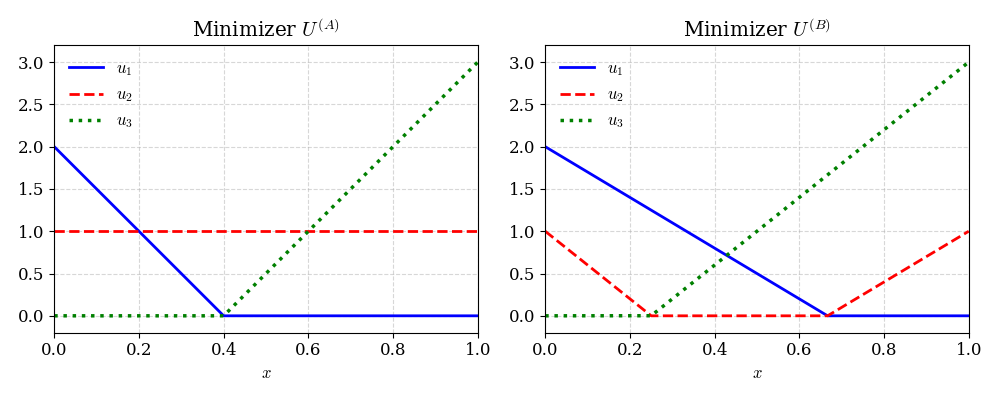}
\vspace{-0.5cm}

\caption{Two distinct minimizers.}
\label{fig:1D-mimimizers}
\end{figure}
\end{example}
Both maps satisfy the boundary conditions and the constraint
$u_1(x)u_2(x)u_3(x)=0$ for every $x\in[0,1]$. Calculating their Dirichlet energies yields
\[
E(U^{(A)}) = \int_0^{2/5} (-5)^2\dx + 0 + \int_{2/5}^1 5^2\dx = 10 + 15 = 25,
\]
and
\[
E(U^{(B)}) = \int_0^{2/3} (-3)^2\dx + \left( \int_0^{1/4} (-4)^2\dx + \int_{2/3}^1 3^2\dx \right) + \int_{1/4}^1 4^2\dx = 6 + (4 + 3) + 12 = 25.
\]
We now show that the global minimum value is $25$. 

Since
$U\in H^1((0,1);\mathbb{R}^3)$, it has an absolutely continuous
representative, and
\[
L(U):=\int_0^1 |U'(x)|\dx
=
\int_0^1
\sqrt{
\bigl(u_1'(x)\bigr)^2+
\bigl(u_2'(x)\bigr)^2+
\bigl(u_3'(x)\bigr)^2
}\dx
\]

By the
Cauchy--Schwarz inequality,
\[
L(U)^2
=
\left(\int_0^1 |U'(x)|\dx\right)^2
\le
\left(\int_0^1 |U'(x)|^2\dx\right)
\left(\int_0^1 1\dx\right)
=E(U).
\]
Therefore, it suffices to prove that $L(U)\ge5$ for every admissible $U$.
Let
\[
K:=\{a\in[0,\infty)^3:a_1a_2a_3=0\}
   =D_1\cup D_2\cup D_3,
\qquad
D_i:=\{a\in[0,\infty)^3:a_i=0\}.
\]
Every admissible $U$ is a continuous curve in $K$ joining
\[
P=(2,1,0)\in D_3
\qquad\text{to}\qquad
Q=(0,1,3)\in D_1.
\]

If the curve meets $D_1\cap D_3$, it passes through some
$R=(0,y,0)$, $y\ge0$, and hence
\[
L(U)\ge |P-R|+|R-Q|
=
\sqrt{4+(1-y)^2}+\sqrt{9+(1-y)^2}
\ge5.
\]

Otherwise, to pass from $D_3$ to $D_1$ it must go through $D_2$,
and hence through points
\[
R=(x_0,0,0)\in D_2\cap D_3,
\qquad
S=(0,0,z_0)\in D_1\cap D_2
\]
for some $x_0,z_0\ge0$. Therefore
\[
L(U)\ge
\sqrt{(2-x_0)^2+1}
+\sqrt{x_0^2+z_0^2}
+\sqrt{1+(3-z_0)^2}.
\]
After unfolding the three coordinate faces into $\mathbb{R}^2$,
the right-hand side is the length of the polygonal path
\[
(2,1)\longrightarrow(x_0,0)\longrightarrow(0,-z_0)
\longrightarrow(-1,-3),
\]
and hence, by the triangle inequality,
\[
L(U)\ge |(2,1)-(-1,-3)|=\sqrt{3^2+4^2}=5.
\]

Thus $E(U)\ge L(U)^2\ge25$ for every admissible $U$. Since
\[
E(U^{(A)})=E(U^{(B)})=25,
\]
the two distinct maps $U^{(A)}$ and $U^{(B)}$ are global minimizers,
which proves nonuniqueness.
end{example}

Relevant numerical strategies include penalty methods, level-set and phase-field descriptions of interfaces \cite{merriman1994,barrett2007}, and augmented-Lagrangian or active-set techniques for nonconvex and complementarity-type constraints \cite{kanzow2021,Wachsmuth2015}. A substantial part of the existing segregation literature concerns pairwise or two-species models; representative examples include \cite{Bozorgnia2009,liu2018numerical,BainesChristou2021,CD2021,CD2022}. These methods do not directly address the overlapping-face geometry of the ternary partial-segregation set.

The distinction is also visible at the differential level. For $m=3$, the derivative of
$g(a)=a_1a_2a_3$ is
$Dg(a)=(a_2a_3,a_1a_3,a_1a_2)$. It is nonzero at a point such as $(0,a_2,a_3)$ with $a_2a_3>0$, but it vanishes whenever at least two components are zero. Moreover, the active derivative direction changes when the identity of the vanishing component changes. Thus the partial-segregation constraint contains regular faces as well as highly degenerate intersections. The squared ternary penalty $(u_1u_2u_3)^2$ inherits this higher-order degeneracy near multiply vanishing states, in addition to the stiffness caused by the factor $\varepsilon^{-1}$. These features justify algorithms tailored to the ternary geometry.

\section{Penalization and  Optimization Algorithms}\label{sec: Penal_algorithms}

Optimization over nonconvex sets occurs throughout applied mathematics, physics, and machine learning. In the present problem, the main issues are nonuniqueness, switching active faces, and stiffness at small $\varepsilon$. We develop two complementary numerical approaches. The first replaces the hard segregation constraint by a smooth product penalty and solves the resulting elliptic system by fixed-point iterations. 
\subsection{Penalization Approach}\label{subsec:penalization}
From now on, $m=3$. For $\varepsilon>0$, define the extended-valued functional
\begin{equation}\label{eq:penalty-functional}
E^\varepsilon(U)
:=\int_\Omega\sum_{i=1}^3|\nabla u_i|^2\dx
+\frac{1}{\varepsilon}\int_\Omega(u_1u_2u_3)^2\dx,
\end{equation}
on the affine trace space
\[
\mathcal{A}_\phi
:=\{U\in(H^1(\Omega))^3:\operatorname{Tr}u_i=\phi_i,
\ i=1,2,3\},
\]
with the convention $E^\varepsilon(U)=+\infty$ if $u_1u_2u_3\notin L^2(\Omega)$. No sign condition is imposed in $\mathcal{A}_\phi$. By Remark~\ref{rem:standing}(i), componentwise truncation reduces the problem to triples taking values in $[0,M_\phi]$, and hence the functional is well defined for the minimizing sequences used below in every finite dimension.
 
\begin{proposition}\label{prop:penalty_conv}
Let Assumption~\ref{ass:boundary_data} hold with $m=3$, and define
\[
m_0:=\inf_{U\in S}E(U),
\qquad
\mathcal M:=\{U\in S:E(U)=m_0\}.
\]
Then the following statements hold.

\begin{enumerate}
\item[(i)]
The constrained problem
\[
\inf_{U\in S}E(U)
\]
admits a global minimizer. For every $\varepsilon>0$, the penalized problem
\[
\inf_{U\in\mathcal A_\phi}E^\varepsilon(U)
\]
also admits a global minimizer $U^\varepsilon$. The minimizers may be
chosen so that
\begin{equation}\label{iq: bdd-minimizers}
    0\le u_i\le M_\phi
\qquad\text{a.e. in }\Omega,
\quad i=1,2,3.
\end{equation}

\item[(ii)]
For every choice of global penalized minimizers 
$\{U^\varepsilon\}_{\varepsilon>0}$, such that \eqref{iq: bdd-minimizers} holds, the family is relatively compact in
$(H^1(\Omega))^3$ as $\varepsilon\downarrow0$ in the following sense:
every sequence $\varepsilon_k\downarrow0$ possesses a subsequence,
not relabeled, and some $\overline U\in\mathcal M$ such that
\[
U^{\varepsilon_k}\longrightarrow\overline U
\qquad\text{strongly in }(H^1(\Omega))^3.
\]
Consequently,
\begin{equation}\label{eq:distance_to_minimizers}
\operatorname{dist}_{(H^1(\Omega))^3}
\bigl(U^\varepsilon,\mathcal M\bigr)
\longrightarrow0
\qquad\text{as }\varepsilon\downarrow0.
\end{equation}

\item[(iii)]
The interaction satisfies
\begin{equation}\label{prod_eps}
\|u_1^\varepsilon u_2^\varepsilon u_3^\varepsilon\|_{L^2(\Omega)}
\le \sqrt{\varepsilon m_0},
\end{equation}
and, along the entire family,
\begin{equation}\label{eq:weighted_penalty_full}
\frac1\varepsilon
\int_\Omega
(u_1^\varepsilon u_2^\varepsilon u_3^\varepsilon)^2\dx
\longrightarrow0.
\end{equation}
Moreover,
\begin{equation}\label{eq:energy_convergence_full}
E(U^\varepsilon)\longrightarrow m_0,
\qquad
E^\varepsilon(U^\varepsilon)\longrightarrow m_0.
\end{equation}

\item[(iv)]
If the constrained minimizer is unique, say
$\mathcal M=\{U^0\}$, then
\[
U^\varepsilon\longrightarrow U^0
\qquad\text{strongly in }(H^1(\Omega))^3
\]
as $\varepsilon\downarrow0$.
\end{enumerate}

No uniqueness of either the constrained or the penalized global minimizer
is assumed or asserted.
\end{proposition}

\begin{proof}
Assumption~\ref{ass:boundary_data} yields $S\neq\emptyset$, and hence
$m_0<\infty$. Let $(U^n)\subset S$ be a minimizing sequence. Since
$u_i^n\ge0$, replacing each component by
$\min\{u_i^n,M_\phi\}$ preserves both the trace condition and the
partial-segregation constraint and does not increase the Dirichlet
energy. Thus we may assume
\[
0\le u_i^n\le M_\phi
\qquad\text{a.e. in }\Omega.
\]

Let $\Phi_i\in H^1(\Omega)$ be a fixed extension of $\phi_i$; for
instance, one may take $\Phi_i=\psi_i$ from
Assumption~\ref{ass:boundary_data}. Since
$u_i^n-\Phi_i\in H_0^1(\Omega)$, Poincar\'e's inequality gives
\[
\|u_i^n\|_{L^2(\Omega)}
\le
\|\Phi_i\|_{L^2(\Omega)}
+
C_P\bigl(
\|\nabla u_i^n\|_{L^2(\Omega)}
+
\|\nabla\Phi_i\|_{L^2(\Omega)}
\bigr).
\]
Hence $(U^n)$ is bounded in $(H^1(\Omega))^3$. After passing to a
subsequence,
\[
U^n\rightharpoonup U^*
\quad\text{in }(H^1(\Omega))^3,
\]
and, by Rellich compactness,
\[
U^n\to U^*
\quad\text{in }(L^2(\Omega))^3
\quad\text{and a.e. in }\Omega.
\]
The trace condition, nonnegativity, and the pointwise identity
$u_1^nu_2^nu_3^n=0$ pass to the limit. Therefore $U^*\in S$.
Weak lower semicontinuity gives
\[
E(U^*)\le\liminf_{n\to\infty}E(U^n)=m_0,
\]
so $U^*\in\mathcal M$.

We next fix $\varepsilon>0$ and consider the penalized problem. Let
$(U^n)\subset\mathcal A_\phi$ be a minimizing sequence. Applying the
componentwise truncation
\[
\mathcal T_{M_\phi}(s)
=
\min\{M_\phi,\max\{s,0\}\}
\]
does not change the trace and does not increase either term in
$E^\varepsilon$. Thus we may again assume
\[
0\le u_i^n\le M_\phi.
\]
Since the segregated extension $\psi$ is an admissible competitor and has
zero interaction term,
\[
\inf_{\mathcal A_\phi}E^\varepsilon
\le E^\varepsilon(\psi)=E(\psi)<\infty.
\]
The preceding Poincar\'e argument therefore yields a uniform
$(H^1(\Omega))^3$ bound on the minimizing sequence. Passing to a subsequence,
\[
U^n\rightharpoonup U^\varepsilon
\quad\text{in }(H^1(\Omega))^3,
\qquad
U^n\to U^\varepsilon
\quad\text{a.e. in }\Omega.
\]
Since
\[
0\le (u_1^nu_2^nu_3^n)^2\le M_\phi^6
\]
and $|\Omega|<\infty$, dominated convergence gives
\[
\int_\Omega
(u_1^nu_2^nu_3^n)^2\dx
\longrightarrow
\int_\Omega
(u_1^\varepsilon u_2^\varepsilon u_3^\varepsilon)^2\dx.
\]
Together with weak lower semicontinuity of the Dirichlet term, this proves
existence of a global minimizer of $E^\varepsilon$.

Let now $U^*\in\mathcal M$. Since $U^*$ is segregated,
\[
E^\varepsilon(U^\varepsilon)
\le E^\varepsilon(U^*)
=E(U^*)=m_0.
\]
Hence
\begin{equation}\label{eq:basic_penalty_bound}
E(U^\varepsilon)\le m_0,
\qquad
\frac1\varepsilon
\|u_1^\varepsilon u_2^\varepsilon u_3^\varepsilon\|_{L^2(\Omega)}^2
\le m_0,
\end{equation}
which proves \eqref{prod_eps} and gives a uniform $(H^1(\Omega))^3$ bound.

Let $\varepsilon_k\downarrow0$ be arbitrary. After extraction,
\[
U^{\varepsilon_k}\rightharpoonup\overline U
\quad\text{in }(H^1(\Omega))^3,
\qquad
U^{\varepsilon_k}\to\overline U
\quad\text{in }(L^2(\Omega))^3
\quad\text{and a.e. in }\Omega.
\]
Estimate \eqref{prod_eps} implies
\[
u_1^{\varepsilon_k}u_2^{\varepsilon_k}u_3^{\varepsilon_k}
\longrightarrow0
\quad\text{in }L^2(\Omega),
\]
and therefore
\[
\overline u_1\overline u_2\overline u_3=0
\qquad\text{a.e. in }\Omega.
\]
Thus $\overline U\in S$. By weak lower semicontinuity and penalized
minimality,
\[
m_0
\le E(\overline U)
\le\liminf_{k\to\infty}E(U^{\varepsilon_k})
\le\limsup_{k\to\infty}E^\varepsilon_k(U^{\varepsilon_k})
\le m_0.
\]
Consequently all inequalities are equalities,
$\overline U\in\mathcal M$, and
\[
E(U^{\varepsilon_k})\longrightarrow E(\overline U)=m_0.
\]
Weak convergence of the gradients together with convergence of their
squared $L^2$ norms gives
\[
\nabla U^{\varepsilon_k}
\longrightarrow\nabla\overline U
\qquad\text{strongly in }L^2(\Omega),
\]
and therefore
\[
U^{\varepsilon_k}\longrightarrow\overline U
\qquad\text{strongly in }(H^1(\Omega))^3.
\]

Since the original sequence $\varepsilon_k\downarrow0$ was arbitrary,
the preceding argument also implies
\[
E(U^\varepsilon)\longrightarrow m_0.
\]
Furthermore,
\[
0
\le
\frac1\varepsilon
\int_\Omega
(u_1^\varepsilon u_2^\varepsilon u_3^\varepsilon)^2\dx
=
E^\varepsilon(U^\varepsilon)-E(U^\varepsilon)
\le
m_0-E(U^\varepsilon),
\]
and hence \eqref{eq:weighted_penalty_full} follows. This also gives
\[
E^\varepsilon(U^\varepsilon)\longrightarrow m_0.
\]

To prove \eqref{eq:distance_to_minimizers}, suppose by contradiction that
there exist $\delta>0$ and a sequence $\varepsilon_k\downarrow0$ such that
\[
\operatorname{dist}_{(H^1(\Omega))^3}
(U^{\varepsilon_k},\mathcal M)\ge\delta.
\]
The compactness argument above provides a subsequence converging strongly
in $(H^1(\Omega))^3$ to some $\overline U\in\mathcal M$, contradicting the last
inequality. Thus \eqref{eq:distance_to_minimizers} holds.

Finally, if $\mathcal M=\{U^0\}$, then
\eqref{eq:distance_to_minimizers} immediately yields
$U^\varepsilon\to U^0$ strongly in $(H^1(\Omega))^3$.

\end{proof}

\begin{proposition}[Euler--Lagrange system for penalized minimizers]
\label{prop:EL_penalized}
Let $U^\varepsilon$ be a global minimizer supplied by
Proposition~\ref{prop:penalty_conv}, chosen so that
\[
0\le u_i^\varepsilon\le M_\phi
\qquad\text{a.e. in }\Omega.
\]
Then, for every $i\in\{1,2,3\}$,
\begin{equation}\label{eq:EL-weak}
\int_\Omega
\nabla u_i^\varepsilon\cdot\nabla\eta\dx
+
\frac1\varepsilon
\int_\Omega
u_i^\varepsilon
\bigl(u_j^\varepsilon u_k^\varepsilon\bigr)^2
\eta\dx
=0
\qquad
\forall\eta\in H_0^1(\Omega),
\end{equation}
where $\{i,j,k\}=\{1,2,3\}$. Equivalently,
\begin{equation}\label{eq:EL-strong}
-\Delta u_i^\varepsilon
+
\frac1\varepsilon
u_i^\varepsilon
\bigl(u_j^\varepsilon u_k^\varepsilon\bigr)^2
=0
\quad\text{in }\Omega,
\qquad
\operatorname{Tr}u_i^\varepsilon=\phi_i,
\end{equation}
in the weak sense.
\end{proposition}

\begin{proof}
Fix $i\in\{1,2,3\}$ and initially let
$\eta\in H_0^1(\Omega)\cap L^\infty(\Omega)$.  Since
$U^\varepsilon$ is bounded, for every $t\in\mathbb R$ the perturbation
\[
U^\varepsilon+t\eta e_i
\]
has the prescribed trace and its ternary product belongs to
$L^2(\Omega)$.  Hence the scalar function
\[
t\longmapsto
E^\varepsilon(U^\varepsilon+t\eta e_i)
\]
is finite in a neighborhood of $t=0$ and has a minimum at $t=0$.
Differentiating there and dividing by $2$ gives
\[
\int_\Omega
\nabla u_i^\varepsilon\cdot\nabla\eta\dx
+
\frac1\varepsilon
\int_\Omega
u_i^\varepsilon
\bigl(u_j^\varepsilon u_k^\varepsilon\bigr)^2
\eta\dx
=0.
\]

Since
\[
u_i^\varepsilon
\bigl(u_j^\varepsilon u_k^\varepsilon\bigr)^2
\in L^\infty(\Omega),
\]
the second term defines a continuous linear functional on
$L^2(\Omega)$ and hence on $H_0^1(\Omega)$.  Because
$H_0^1(\Omega)\cap L^\infty(\Omega)$ is dense in $H_0^1(\Omega)$,
the identity extends to every $\eta\in H_0^1(\Omega)$.
\end{proof}

For fixed $u_j^\varepsilon$ and $u_k^\varepsilon$, equation
\eqref{eq:EL-strong} is linear in $u_i^\varepsilon$:
\begin{equation}\label{eq:linearized}
-\Delta u_i^\varepsilon+c_i^\varepsilon(x)u_i^\varepsilon=0,
\qquad
c_i^\varepsilon(x)
:=
\frac1\varepsilon
\bigl(u_j^\varepsilon u_k^\varepsilon\bigr)^2
\in L^\infty(\Omega),
\qquad
c_i^\varepsilon\ge0.
\end{equation}
The weak maximum principle, proved below in
Lemma~\ref{lem:1.2}, then yields
\[
0\le u_i^\varepsilon
\le\|\phi_i\|_{L^\infty(\partial\Omega)}
\qquad\text{a.e. in }\Omega.
\]
\begin{remark}
\label{rem:minimizer_vs_fixedpoint}
Every bounded global minimizer of $E^\varepsilon$ is a weak solution of
\eqref{eq:system}.  Conversely, a weak solution of \eqref{eq:system} need
not be a global minimizer of the nonconvex energy $E^\varepsilon$.
Likewise, the fixed-point iterations introduced below are designed to
compute weak solutions of \eqref{eq:system}; unless an additional
variational selection property is established, their limits should not be
identified with global minimizers.

Accordingly, Proposition~\ref{prop:penalty_conv} concerns global
penalized minimizers, whereas the fixed-$\varepsilon$ iteration results
below concern solutions or fixed points of \eqref{eq:system}.  
\end{remark}
 
Proposition~\ref{prop:penalty_conv} gives global variational convergence for global minimizers. A different issue is the characterization of arbitrary local minimizers or accumulation points generated by a nonconvex numerical algorithm. At regular points of one face the product constraint has a nonzero derivative, whereas at intersections of two or more faces it is degenerate and the active face may switch spatially. Developing a function-space stationarity system that remains meaningful across these intersections would require tools related to complementarity-constrained optimization in Banach spaces; see \cite{Wachsmuth2015}. We do not pursue that theory here.

For the fixed-$\varepsilon$ analysis, let $\Omega\subset\mathbb{R}^d$ be a bounded Lipschitz domain and let
$\phi_i\in H^{1/2}(\partial\Omega)\cap L^\infty(\partial\Omega)$ satisfy $\phi_i\ge0$, $i=1,2,3$. We suppress the superscript $\varepsilon$ whenever $\varepsilon$ is fixed and consider
\begin{equation}\label{eq:system}
\begin{cases}
-\Delta u_i+\dfrac{1}{\varepsilon}u_i\displaystyle\prod_{j\ne i}u_j^2=0
& \text{in }\Omega,\\[1ex]
\operatorname{Tr}u_i=\phi_i
& \text{on }\partial\Omega.
\end{cases}
\qquad i=1,2,3.
\end{equation}

For comparison with the local singular-limit theory, we also record a
consequence of the results of Soave and Terracini
\cite{soave2024partial1}. 

\begin{definition}\label{def: loc_var_select}
A family
$(U^\varepsilon)_{\varepsilon>0}$ will be called
\emph{locally variationally selected} if, for every
$\omega\Subset\Omega$, the family is uniformly bounded in
$L^\infty(\omega)$ and, for all sufficiently small $\varepsilon$,
\[
E^\varepsilon(U^\varepsilon;\omega)
\le
E^\varepsilon(U^\varepsilon+\Phi;\omega)
\qquad
\text{for every }
\Phi\in H_0^1(\omega;\mathbb R^3),
\]
where
\[
E^\varepsilon(U;\omega)
:=
\int_\omega\sum_{i=1}^3|\nabla u_i|^2\dx
+
\frac1\varepsilon
\int_\omega(u_1u_2u_3)^2\dx.
\]
\end{definition}

\begin{lemma}[Local vanishing of the interaction energy]
\label{lem:B}
Let $(U^\varepsilon)$ be locally variationally selected. Then for every
sequence $\varepsilon_n\downarrow0$ there exists a subsequence, not
relabeled, such that
\[
frac1{\varepsilon_n}
\int_\omega
(u_1^{\varepsilon_n}
  u_2^{\varepsilon_n}
  u_3^{\varepsilon_n})^2\dx
\longrightarrow0
\]
for every open set $\omega\Subset\Omega$.
\end{lemma}

\begin{proof}
Set $\beta_n=\varepsilon_n^{-1}$.  On every
$\omega'\Subset\Omega$, the local $L^\infty$ bound and compact-support
minimality give hypotheses \textup{(h1)}--\textup{(h2)} of
\cite[Theorem~1.1]{soave2024partial1}, after restricting to a slightly
larger subdomain containing $\overline\omega$.  Therefore, after
extraction,
\[
\beta_n
\int_\omega
\prod_{j=1}^3(u_j^{\varepsilon_n})^2\dx
\longrightarrow0
\qquad
\text{for every }\omega\Subset\Omega,
\]
which is the claimed conclusion.
\end{proof}


\begin{lemma}[{\cite[Theorems~1.2 and~1.7]{soave2024partial1}}]\label{lem:C}
Assume that $\Omega$ has $C^1$ boundary and that the boundary data admit a nonnegative Lipschitz extension
$\psi\in\operatorname{Lip}(\overline\Omega)^3$ satisfying
$\operatorname{Tr}\psi_i=\phi_i$ and
\[
\psi_1\psi_2\psi_3\equiv0
\quad\text{in }\overline\Omega.
\]
Let  $U^\varepsilon$ be locally variationally selected as in Definition \ref{def: loc_var_select} such that  
\[
U^\varepsilon\in
\operatorname*{argmin}_{U\in\mathcal A_\phi}E^\varepsilon(U).
\]
Then there exists
$\bar\nu\in(0,1)$, depending only on the dimension, such that, for every
$\alpha\in(0,\bar\nu)$,
\[
\|U^\varepsilon\|_{C^{0,\alpha}(\overline\Omega)}\le C_\alpha
\qquad(0<\varepsilon<1),
\]
where $C_\alpha$ is independent of $\varepsilon$. Moreover, along a sequence
$\varepsilon_k\downarrow0$, there exists
$U=(u_1,u_2,u_3)$ such that
\[
U^{\varepsilon_k}\to U
\quad\text{strongly in }H^1(\Omega;\mathbb R^3)
\quad\text{and in }C^{0,\alpha}(\overline\Omega;\mathbb R^3)
\]
for every $\alpha\in(0,\bar\nu)$. The limit is a global minimizer of the constrained problem and satisfies
\[
u_1u_2u_3\equiv0\quad\text{in }\Omega,
\qquad
\Delta u_i=0\quad\text{in }\{u_i>0\},
\quad i=1,2,3.
\]
\end{lemma}
\begin{remark}\label{rem:regularity_scope}
The exponent in Lemma~\ref{lem:C} is
$\bar\nu=\alpha_{3,d}/3\le2/3$; see
\cite[Remark~1.3]{soave2024partial1}.  The companion work
\cite[Theorem~1.6]{soave2024partial2} improves the uniform
\emph{interior} estimates to every $\alpha<3/4$, while limiting minimizing
maps are $C^{0,3/4}$ in the interior and this exponent is optimal in
general.

We retain Lemma~\ref{lem:C} in its global, up-to-the-boundary form because
that is the result supplied by
\cite[Theorem~1.2]{soave2024partial1}.  Proposition~\ref{prop:penalty_conv}
already provides the global $H^1$ compactness needed in the present paper;
the additional information supplied by Lemma~\ref{lem:C} is the uniform
H\"older control and the harmonicity of the limiting components on their
positivity sets.

The $C^1$ boundary hypothesis of Lemma~\ref{lem:C} does not apply directly
co the square domains used in Section~\ref{sec: simulations}.  Accordingly, this lemma is used
only as analytical background for the singular limit and is not invoked
in the fixed-$\varepsilon$ iteration analysis or in the justification of
the numerical experiments.
\end{remark}
\subsection{Auxiliary Lemmas}
We prove several technical lemmas, which will play an important role in our further analysis. For the regularized system \eqref{eq:system} we set
\begin{equation}\label{eq:Mdef}
M:=\max_{i=1,2,3}\|\phi_i\|_{L^\infty(\partial\Omega)}.
\end{equation}
Let $\lambda_\Omega>0$ be the first Dirichlet eigenvalue of $-\Delta$ on $\Omega$:
\begin{equation}\label{eq:lambdaD}
\lambda_\Omega
:=\inf_{w\in H_0^1(\Omega)\setminus\{0\}}
\frac{\int_\Omega|\nabla w|^2\dx}
{\int_\Omega|w|^2\dx}.
\end{equation}
Thus
\begin{equation}\label{eq:poincare}
\int_\Omega|\nabla w|^2\dx
\ge\lambda_\Omega\int_\Omega|w|^2\dx
\qquad\text{for all }w\in H_0^1(\Omega).
\end{equation}
We repeatedly use the elementary inequalities
\begin{equation}\label{eq:ineq3}
|a^2-b^2|\le2M|a-b|,
\qquad
|(ab)^2-(cd)^2|
\le2M^3\bigl(|a-c|+|b-d|\bigr),
\end{equation}
valid for $a,b,c,d\in[-M,M]$.

We begin with the linear Dirichlet theory used throughout: well-posedness, an $H^1$ estimate, and the weak maximum principle.

\begin{lemma}\label{lem:1.1}
Let $q\in L^\infty(\Omega)$ with $q\ge0$, and let
$\phi\in H^{1/2}(\partial\Omega)$. There exists a unique
$z\in H^1(\Omega)$ with trace $\operatorname{Tr}z=\phi$ such that
\begin{equation}\label{eq:4}
-\Delta z+\frac{q}{\varepsilon}z=0
\quad\text{in }\Omega
\end{equation}
in the weak sense. Moreover,
\begin{equation}\label{eq:5}
\|z\|_{H^1(\Omega)}
\le C(\Omega)
\left(1+\sqrt{\frac{\|q\|_{L^\infty(\Omega)}}{\varepsilon}}\right)
\|\phi\|_{H^{1/2}(\partial\Omega)},
\end{equation}
where $C(\Omega)>0$ is independent of $q$, $\varepsilon$, and $\phi$.
\end{lemma}

\begin{proof}
On a bounded Lipschitz domain, the trace operator
$H^1(\Omega)\to H^{1/2}(\partial\Omega)$ has a bounded right inverse; see \cite{Necas2012}. Hence there exists an extension $\Phi\in H^1(\Omega)$ of $\phi$ satisfying
\[
\|\Phi\|_{H^1(\Omega)}
\le C(\Omega)\|\phi\|_{H^{1/2}(\partial\Omega)}.
\]
Write $z=\Phi+w$ with $w\in H_0^1(\Omega)$. The weak equation for $w$ is
\[
\int_\Omega\nabla w\cdot\nabla\psi\dx
+\int_\Omega\frac{q}{\varepsilon}w\psi\dx
=-\int_\Omega\nabla\Phi\cdot\nabla\psi\dx
-\int_\Omega\frac{q}{\varepsilon}\Phi\psi\dx
\]
for every $\psi\in H_0^1(\Omega)$. The left-hand side is continuous and coercive on $H_0^1(\Omega)$, so Lax--Milgram gives a unique $w$, and therefore a unique $z$.

Taking $\psi=w$ and applying Cauchy--Schwarz and Young's inequalities gives
\[
\frac12\|\nabla w\|_{L^2}^2
+\frac12\int_\Omega\frac{q}{\varepsilon}w^2\dx
\le
\frac12\|\nabla\Phi\|_{L^2}^2
+\frac12\int_\Omega\frac{q}{\varepsilon}\Phi^2\dx.
\]
Consequently,
\[
\|\nabla w\|_{L^2}^2
\le
\left(1+\frac{\|q\|_{L^\infty}}{\varepsilon}\right)
\|\Phi\|_{H^1}^2.
\]
Since $w\in H_0^1(\Omega)$, Poincar\'e's inequality implies
$\|w\|_{H^1}\le C(\Omega)\\|\nabla w\|_{L^2}$. Combining this estimate with $z=\Phi+w$ and the bound for the extension yields \eqref{eq:5}.
\end{proof}

The next lemma gives uniform pointwise bounds from the boundary data.

\begin{lemma}\label{lem:1.2}
Let $q\in L^\infty(\Omega)$ with $q\ge0$, and let $z\in H^1(\Omega)$ be a weak solution of
\[
-\Delta z+\frac{q}{\varepsilon}z=0
\quad\text{in }\Omega
\]
with trace $\operatorname{Tr}z=\phi$, where
$\phi\in H^{1/2}(\partial\Omega)\cap L^\infty(\partial\Omega)$ and $\phi\ge0$. Then
\[
0\le z\le\|\phi\|_{L^\infty(\partial\Omega)}
\qquad\text{a.e. in }\Omega.
\]
\end{lemma}

\begin{proof}
Let $m:=\|\phi\|_{L^\infty(\partial\Omega)}$ and
$z^-:=(-z)^+$. Since the trace of $z^-$ vanishes, $z^-\in H_0^1(\Omega)$. Testing the weak equation with $-z^-$ gives
\[
\int_\Omega|\nabla z^-|^2\dx
+\int_\Omega\frac{q}{\varepsilon}(z^-)^2\dx=0.
\]
Hence $z^-=0$ and $z\ge0$ a.e.

Now let $w:=(z-m)^+$. Since $\phi-m\le0$ on the boundary in the trace sense, $w\in H_0^1(\Omega)$. Testing with $w$ gives
\[
\int_\Omega|\nabla w|^2\dx
+\int_\Omega\frac{q}{\varepsilon}zw\dx=0.
\]
Both terms are nonnegative because $z\ge0$, so $w=0$. Therefore $z\le m$ a.e.
\end{proof}

We next compare solutions corresponding to ordered coefficients.

\begin{lemma}\label{lem:1.3}
Let $c_1,c_2\in L^\infty(\Omega)$ satisfy $0\le c_1\le c_2$ a.e. in $\Omega$. Let $u,v\in H^1(\Omega)$ be weak solutions of
\[
-\Delta u+\frac{c_1}{\varepsilon}u=0,
\qquad
-\Delta v+\frac{c_2}{\varepsilon}v=0
\quad\text{in }\Omega,
\]
with the same trace
$\operatorname{Tr}u=\operatorname{Tr}v=\phi$, where
$\phi\in H^{1/2}(\partial\Omega)\cap L^\infty(\partial\Omega)$ and $\phi\ge0$. Then
\[
u\ge v\qquad\text{a.e. in }\Omega.
\]
\end{lemma}

\begin{proof}
Lemma~\ref{lem:1.2} gives $u\ge0$ a.e. in $\Omega$. Set
$w:=(v-u)^+$. Since $u$ and $v$ have the same trace,
$w\in H_0^1(\Omega)$. Subtracting the two equations yields
\[
-\Delta(v-u)+\frac{c_2}{\varepsilon}(v-u)
=-\frac{c_2-c_1}{\varepsilon}u\le0
\]
in the weak sense. Testing this inequality with $w$ gives
\[
\int_\Omega |\nabla w|^2\dx
+\int_\Omega\frac{c_2}{\varepsilon}w^2\dx
=-\int_\Omega\frac{c_2-c_1}{\varepsilon}uw\dx\le0.
\]
Both terms on the left are nonnegative, and therefore $w=0$ a.e. Thus $v\le u$ a.e. in $\Omega$.
\end{proof}

We also record an $L^2$ resolvent estimate for the homogeneous Dirichlet problem. It will be used repeatedly for difference equations.

\begin{lemma}\label{lem:1.4}
Let $q\in L^\infty(\Omega)$ satisfy $q\ge0$, let $f\in L^2(\Omega)$, and let $w\in H_0^1(\Omega)$ solve
\[
-\Delta w+\frac{q}{\varepsilon}w=f
\quad\text{in }\Omega
\]
weakly. Then
\[
\|w\|_{L^2(\Omega)}
\le \frac{1}{\lambda_\Omega}\|f\|_{L^2(\Omega)}.
\]
\end{lemma}

\begin{proof}
Testing the equation with $w$ gives
\[
\int_\Omega|\nabla w|^2\dx
+\int_\Omega\frac{q}{\varepsilon}w^2\dx
=\int_\Omega fw\dx
\le\|f\|_{L^2}\|w\|_{L^2}.
\]
After dropping the nonnegative zeroth-order term and applying
\eqref{eq:poincare}, we obtain
\[
\lambda_\Omega\|w\|_{L^2}^2
\le\|f\|_{L^2}\|w\|_{L^2}.
\]
The conclusion follows immediately, with the case $w=0$ being trivial. The estimate is uniform in both $q$ and $\varepsilon$.
\end{proof}

The next results describe two distinct features of the strong-competition regime $\varepsilon\downarrow0$. The first is a local exponential decay estimate in regions where a competing pair is uniformly positive. The second concerns variationally selected families and the vanishing of their weighted interaction energy. In the notation of \cite{soave2024partial1}, this regime corresponds to $\beta=\varepsilon^{-1}\to+\infty$.

\begin{lemma}\label{lem:A}
Let $q\in L^\infty(\Omega)$ satisfy $q\ge0$ a.e. in $\Omega$, and let $z\in H^1(\Omega)$ be a weak solution of
\[
-\Delta z+\frac{q}{\varepsilon}z=0
\quad\text{in }\Omega
\]
with trace $\operatorname{Tr}z=\phi$, where
$\phi\in H^{1/2}(\partial\Omega)\cap L^\infty(\partial\Omega)$ and $\phi\ge0$. Let $K_0\Subset U\Subset\Omega$, with $K_0$ compact and $U$ open, and assume that
\[
q\ge q_0>0\qquad\text{a.e. in }U.
\]
Then there exist constants $C,c>0$, depending only on $d$ and on
$\operatorname{dist}(K_0,\partial U)$, such that
\[
\|z\|_{L^\infty(K_0)}
\le C\|\phi\|_{L^\infty(\partial\Omega)}
\exp\!\left(-c\sqrt{\frac{q_0}{\varepsilon}}\right).
\]
Here and below $z$ is understood through its locally continuous representative.
\end{lemma}

\begin{proof}
By Lemma~\ref{lem:1.2},
\[
0\le z\le A:=\|\phi\|_{L^\infty(\partial\Omega)}
\qquad\text{a.e. in }\Omega.
\]
Since $qz\in L^\infty_{\mathrm{loc}}(\Omega)$, interior elliptic regularity gives a locally continuous representative of $z$. Choose
\[
r:=\frac14\operatorname{dist}(K_0,\partial U)>0.
\]
For each $x_0\in K_0$, one has $B_{2r}(x_0)\Subset U$, and in this ball
\[
-\Delta z\le-\frac{q_0}{\varepsilon}z.
\]
Let $m:=\sqrt{q_0/\varepsilon}$ and let $h$ solve
\[
-\Delta h+m^2h=0\quad\text{in }B_{2r}(x_0),
\qquad h=1\quad\text{on }\partial B_{2r}(x_0).
\]
    Since
\[
\left(-\Delta+\frac q\varepsilon\right)(Ah)
=
A\frac{q-q_0}{\varepsilon}h
\ge0,
\]
while
\[
\left(-\Delta+\frac q\varepsilon\right)z=0,
\]
and $z\le A=Ah$ on $\partial B_{2r}(x_0)$, the weak comparison principle
gives
\[
0\le z\le Ah
\qquad\text{in }B_{2r}(x_0).
\]

The standard radial barrier estimate for the modified Helmholtz equation gives
\[
\sup_{B_r(x_0)}h
\le C_d e^{-c_dmr};
\]
see, for example, \cite[Lemma~4.4]{MR2146353}. Consequently,
\[
\|z\|_{L^\infty(B_r(x_0))}
\le C_d A e^{-c_dmr}.
\]
Taking the supremum over $x_0\in K_0$ and absorbing the fixed distance $r$ into the exponential constant proves the claim.
\end{proof}

Define the Banach space
\[
X:=(L^2(\Omega))^3,
\qquad
\|U\|_X:=\max_{i=1,2,3}\|u_i\|_{L^2(\Omega)},
\]
and the closed, bounded, convex set
\[
K:=\{U\in X:0\le u_i\le M\text{ a.e. in }\Omega,
\ i=1,2,3\}.
\]

\subsection*{\texorpdfstring{Jacobi/Picard map $T$}{Jacobi/Picard map T}}

For $U=(u_1,u_2,u_3)\in K$, define $T(U)=V$ by requiring $v_i$ to solve
\begin{equation}\label{eq:8}
\begin{cases}
-\Delta v_i+\dfrac{(u_ju_\ell)^2}{\varepsilon}v_i=0
&\text{in }\Omega,\\[1ex]
\operatorname{Tr}v_i=\phi_i
&\text{on }\partial\Omega,
\end{cases}
\qquad \{i,j,\ell\}=\{1,2,3\}.
\end{equation}
Because $0\le(u_ju_\ell)^2\le M^4$, Lemma~\ref{lem:1.1} shows that $T$ is well defined. Lemma~\ref{lem:1.2} will show below that $T(K)\subset K$.

\subsection*{\texorpdfstring{Sequential semi-implicit map $G$}{Sequential semi-implicit map G}}

For $U\in K$, define $G(U)=V$ sequentially by
\begin{equation}\label{eq:9}
\begin{cases}
-\Delta v_1+\dfrac{(u_2u_3)^2}{\varepsilon}v_1=0,
&\operatorname{Tr}v_1=\phi_1,\\[1ex]
-\Delta v_2+\dfrac{u_3^2(u_1^2+v_1^2)}{2\varepsilon}v_2=0,
&\operatorname{Tr}v_2=\phi_2,\\[1ex]
-\Delta v_3+\dfrac{(u_1u_2)^2+(v_1v_2)^2}{2\varepsilon}v_3=0,
&\operatorname{Tr}v_3=\phi_3.
\end{cases}
\end{equation}
This is a Gauss--Seidel-type sequential map, but it is not the classical Gauss--Seidel map associated with $T$: the second and third coefficients use arithmetic averages of old and newly computed squared products.

If $U=G(U)$, then the averaged coefficients reduce to
\[
\frac{u_3^2(u_1^2+u_1^2)}{2}=u_1^2u_3^2,
\qquad
\frac{(u_1u_2)^2+(u_1u_2)^2}{2}=(u_1u_2)^2.
\]
Hence, the fixed points of $G$ are precisely weak solutions of
\eqref{eq:system}.  Thus the averaging in \eqref{eq:9} changes the
iteration, but not its fixed-point equations.
\subsection{Numerical Algorithms for Penalized System}\label{subsec: Pickard_iter}
In this section we present numerical algorithms to solve the penalized system \eqref{eq:system}. We first consider a parallel, relaxed Jacobi/Picard scheme
that decouples the three-component system into independent linear elliptic updates.
\begin{algorithm}[H]
\caption{Relaxed Jacobi/Picard Iteration}
\label{alg:monotone_iterative}
\begin{algorithmic}[1]
\State \textbf{Input:} $\varepsilon>0$, relaxation parameter $\omega\in(0,1]$, tolerance $\mathrm{tol}>0$, and a maximum number of iterations.
\State \textbf{Initialize:} For $i=1,2,3$, compute the harmonic extension $u_i^0$ from
$-\Delta u_i^0=0$ in $\Omega$ and $\operatorname{Tr}u_i^0=\phi_i$.
\For{$k=0,1,2,\ldots$ until termination}
    \State For $i=1,2,3$, solve in parallel
    \[
    -\Delta \widehat u_i^{k+1}
    +\frac{1}{\varepsilon}
    \left(\prod_{j\ne i}u_j^k\right)^2
    \widehat u_i^{k+1}=0
    \quad\text{in }\Omega,
    \qquad
    \operatorname{Tr}\widehat u_i^{k+1}=\phi_i.
    \]
    \State Set
    \begin{equation}\label{eq:monotone}
    u_i^{k+1}
    =(1-\omega)u_i^k+\omega\widehat u_i^{k+1},
    \qquad i=1,2,3.
    \end{equation}
    \State Stop if $\|U^{k+1}-U^k\|_X\le\mathrm{tol}$.
\EndFor
\end{algorithmic}
\end{algorithm}
Algorithm~\ref{alg:monotone_iterative} is a relaxed Jacobi/Picard iteration:
\[
U^{k+1}=(1-\omega)U^k+\omega T(U^k).
\]
For $\omega=1$, the map $T$ is order reversing by Lemma~\ref{lem:1.3}; starting from the harmonic extension, this produces monotone even and odd subsequences. That alternating monotonicity does not follow for a general relaxed step $0<\omega<1$. Moreover, the global estimate in Lemma~\ref{lem:1.8} gives
\[
\|[(1-\omega)I+\omega T](U)-[(1-\omega)I+\omega T](V)\|_X
\le(1-\omega+\omega L_T)\|U-V\|_X.
\]
Thus this particular contraction estimate requires $L_T<1$; damping changes the rate but does not make the displayed worst-case bound contractive when $L_T\ge1$.

\subsubsection{Implementation Details}\label{rem:implementation}
Let $\Omega$ be polygonal and let $\mathcal T_h$ be a shape-regular triangulation with nodes $x_1,\ldots,x_N$ and continuous piecewise-affine basis functions $\varphi_1,\ldots,\varphi_N$.  For the discrete implementation we assume that the prescribed boundary data are given by
pointwise evaluable representatives;
the $H^{1/2}$ hypothesis  is the analytical trace assumption.  Split the indices into interior nodes $I$ and Dirichlet boundary nodes $B$, and write
\[
u_i=
\begin{bmatrix}u_{i,I}\\ g_{i,B}\end{bmatrix},
\qquad
(g_{i,B})_r=\phi_i(x_r),\quad r\in B.
\]

The stiffness and consistent mass matrices are
\[
K_{jr}=\int_\Omega\nabla\varphi_j\cdot\nabla\varphi_r\dx,
\qquad
M_{jr}=\int_\Omega\varphi_j\varphi_r\dx.
\]
Let
\[
m_j:=\int_\Omega\varphi_j\dx,
\qquad
M^L:=\operatorname{diag}(m_1,\ldots,m_N).
\]
For a nonnegative nodal coefficient vector $w$, the lumped weighted mass matrix is
\[
M^L(w):=\operatorname{diag}(m_1w_1,\ldots,m_Nw_N)
=M^L\operatorname{diag}(w).
\]
The full matrix $K$ is symmetric positive semidefinite, while the Dirichlet block $K_{II}$ is symmetric positive definite. The matrices $M$, $M^L$, and $M^L(w)$ are symmetric, with the first two positive definite and the last positive semidefinite.

For the Jacobi/Picard iteration, define the nodal weights
\[
w_1^k=(u_2^k\odot u_3^k)^2,
\qquad
w_2^k=(u_1^k\odot u_3^k)^2,
\qquad
w_3^k=(u_1^k\odot u_2^k)^2,
\]
where $\odot$ and the square are understood componentwise. Each unrelaxed update is obtained from
\begin{equation}\label{eq:interior-system}
\left(K_{II}+\frac1\varepsilon M^L_{II}(w_i^k)\right)
\widehat u_{i,I}^{k+1}
=-K_{IB}g_{i,B},
\qquad
\widehat u_{i,B}^{k+1}=g_{i,B}.
\end{equation}
There is no boundary contribution from the weighted term because the lumped matrix is diagonal. The relaxed interior vector is then
\[
u_{i,I}^{k+1}
=(1-\omega)u_{i,I}^k+\omega\widehat u_{i,I}^{k+1}.
\]
The harmonic initialization solves
\begin{equation}\label{eq:interior-system_harmonic}
K_{II}u_{i,I}^0=-K_{IB}g_{i,B}.
\end{equation}
For $T$, the three systems in \eqref{eq:interior-system} are independent and may be solved in parallel. For the sequential map $G$, the same matrix form is used, but the second and third nodal weights are replaced by the averaged coefficients in \eqref{eq:9} and the systems must be solved in order.

Every matrix in \eqref{eq:interior-system} is symmetric positive definite. Its sparsity pattern is fixed, so symbolic ordering and symbolic factorization data can be reused, although the numerical factorization changes whenever the diagonal weight changes. Sparse Cholesky is suitable for the moderate grids used in Section~\ref{sec: simulations}. On large meshes, preconditioned conjugate gradients can be preferable; near-linear complexity requires a preconditioner that is robust with respect to the mesh and the possibly high-contrast reaction coefficient, and is therefore not automatic in the limit $\varepsilon\downarrow0$.

On a uniform square grid of spacing $h$, split each square into right triangles using a consistent diagonal and apply mass lumping. The interior stiffness stencil is the standard five-point stencil, while the lumped nodal mass is $h^2$. Dividing the finite-element equations by $h^2$ therefore gives the classical five-point finite-difference discretization of
$-\Delta u_i+(w_i/\varepsilon)u_i=0$ with nodal evaluation of $w_i$. This is the precise sense in which the finite-difference implementation used in Section~\ref{sec: simulations} agrees with the lumped $P_1$ formulation.

\subsection{Analytical Properties of the Fixed-Point Maps $T$ and $G$}\label{subsec:fixed_point_properties}

We next establish the invariant-set, compactness, continuity, and
Lipschitz properties of the fixed-point maps introduced above. These
results provide the analytical basis for the existence and convergence
statements proved subsequently.

\begin{lemma}\label{lem:1.5}
For every $U\in K$, one has $T(U)\in K$ and $G(U)\in K$.
\end{lemma}

\begin{proof}
Each equation in \eqref{eq:8} and \eqref{eq:9} has the form
\[
-\Delta v_i+\frac{q_i}{\varepsilon}v_i=0
\]
with nonnegative boundary data and a coefficient satisfying $0\le q_i\le M^4$. For $G$, this assertion is applied sequentially: the estimate for $v_1$ gives $0\le v_1\le M$, which makes the second coefficient bounded by $M^4$; the estimate for $v_2$ then gives the corresponding bound for the third coefficient. Lemma~\ref{lem:1.2} therefore yields
\[
0\le v_i\le\|\phi_i\|_{L^\infty(\partial\Omega)}\le M
\qquad\text{a.e. in }\Omega,
\]
for all three components and for both maps.
\end{proof}

We next establish the compactness and continuity properties required by Schauder's fixed-point theorem.

\begin{lemma}\label{lem:1.6}
The map $T:K\to K\subset X$ is compact; equivalently, $T(K)$ is relatively compact in $X$.
\end{lemma}

\begin{proof}
Let $V=T(U)$. For each component, Lemma~\ref{lem:1.1}, with
$q_i=(u_ju_\ell)^2$, gives
\[
\|v_i\|_{H^1(\Omega)}
\le C(\Omega)
\left(1+\frac{M^2}{\sqrt\varepsilon}\right)
\|\phi_i\|_{H^{1/2}(\partial\Omega)}.
\]
For fixed $\varepsilon>0$, the right-hand side is uniform over $U\in K$. Thus $T(K)$ is bounded in $(H^1(\Omega))^3$. The compact embedding
$H^1(\Omega)\hookrightarrow L^2(\Omega)$ on bounded Lipschitz domains implies that $T(K)$ is relatively compact in $X$.
\end{proof}

\begin{lemma}\label{lem:1.7}
The map $T:K\to X$ is continuous.
\end{lemma}

\begin{proof}
Let $U^{(n)}\to U$ in $X$, with $U^{(n)},U\in K$, and set
$V^{(n)}=T(U^{(n)})$, $V=T(U)$. Fix $i$ and write
\[
q^{(n)}=(u_j^{(n)}u_\ell^{(n)})^2,
\qquad
q=(u_ju_\ell)^2.
\]
By \eqref{eq:ineq3},
\[
\|q^{(n)}-q\|_{L^2}
\le2M^3\left(
\|u_j^{(n)}-u_j\|_{L^2}
+\|u_\ell^{(n)}-u_\ell\|_{L^2}
\right)
\longrightarrow0.
\]
The difference $w^{(n)}:=v_i^{(n)}-v_i$ belongs to $H_0^1(\Omega)$ and solves
\[
-\Delta w^{(n)}+\frac{q^{(n)}}{\varepsilon}w^{(n)}
=\frac{q-q^{(n)}}{\varepsilon}v_i.
\]
Using Lemmas~\ref{lem:1.4} and~\ref{lem:1.5}, we obtain
\[
\|w^{(n)}\|_{L^2}
\le\frac{M}{\varepsilon\lambda_\Omega}
\|q^{(n)}-q\|_{L^2}
\longrightarrow0.
\]

Taking the maximum over the components proves continuity in $X$.
\end{proof}

\begin{lemma}\label{lem:1.8}
The Picard map $T$ is globally Lipschitz on $K$:
\[
\|T(U)-T(V)\|_X
\le L_T\|U-V\|_X,
\qquad
L_T:=\frac{4M^4}{\varepsilon\lambda_\Omega}.
\]
Moreover, let $K_0\Subset U_0\Subset\Omega$, where $K_0$ is compact and $U_0$ is open. Fix $i\in\{1,2,3\}$ and write
$\{i,j,\ell\}=\{1,2,3\}$. If there exists $\delta>0$ such that
\[
(u_ju_\ell)^2\ge\delta,
\qquad
(v_jv_\ell)^2\ge\delta
\quad\text{a.e. in }U_0,
\]
then
\[
\|T(U)_i-T(V)_i\|_{L^2(K_0)}
\le CM\exp\!\left(-c\sqrt{\frac\delta\varepsilon}\right),
\]
where $C,c>0$ depend only on $d$, $K_0$, and $U_0$.
\end{lemma}

\begin{proof}
Fix $i$ and set
$Z_U:=T(U)_i$, $Z_V:=T(V)_i$,
$q_U:=(u_ju_\ell)^2$, and $q_V:=(v_jv_\ell)^2$. The difference
$w:=Z_U-Z_V\in H_0^1(\Omega)$ satisfies
\[
-\Delta w+\frac{q_U}{\varepsilon}w
=\frac{q_V-q_U}{\varepsilon}Z_V.
\]
Lemmas~\ref{lem:1.4} and~\ref{lem:1.5} imply
\[
\|w\|_{L^2}
\le\frac{M}{\varepsilon\lambda_\Omega}
\|q_V-q_U\|_{L^2}.
\]
Using \eqref{eq:ineq3},
\[
\|q_V-q_U\|_{L^2}
\le2M^3\left(
\|v_j-u_j\|_{L^2}
+\|v_\ell-u_\ell\|_{L^2}
\right)
\le4M^3\|U-V\|_X.
\]
This proves the global estimate after taking the maximum over $i$.

Under the additional interior hypothesis, Lemma~\ref{lem:A} applies separately to $Z_U$ and $Z_V$ and gives
\[
\|Z_U\|_{L^\infty(K_0)}
+\|Z_V\|_{L^\infty(K_0)}
\le CM\exp\!\left(-c\sqrt{\frac\delta\varepsilon}\right).
\]
Therefore
\[
\|Z_U-Z_V\|_{L^2(K_0)}
\le |K_0|^{1/2}
\left(
\|Z_U\|_{L^\infty(K_0)}
+\|Z_V\|_{L^\infty(K_0)}
\right),
\]
which yields the claimed bound after changing $C$. This estimate is absolute rather than proportional to $\|U-V\|_X$: both outputs are exponentially pinned to zero in the region where their corresponding potentials are uniformly positive.
\end{proof}

\begin{remark}\label{rem:1.1}
If the lower bound in Lemma~\ref{lem:1.8} depends on $\varepsilon$, the same argument gives the factor
\[
\exp\!\left(-c\sqrt{\frac{\delta_\varepsilon}{\varepsilon}}\right).
\]
It tends to zero only when $\delta_\varepsilon/\varepsilon\to+\infty$. In particular, a lower bound of order $O(\varepsilon)$ does not produce an asymptotically vanishing exponential factor.
\end{remark}

\begin{remark}\label{rem:1.2}
The constant $L_T$ is a worst-case estimate over the full order interval $K$ and deteriorates as $\varepsilon\downarrow0$. Lemma~\ref{lem:A} gives a genuinely sharper local statement in regions where two competing components are uniformly positive. Lemma~\ref{lem:B}, by contrast, concerns only variationally selected solution families and controls the weighted ternary product in an integral sense. By itself it does not improve the global Lipschitz constant of $T$ or $G$, nor does it provide a pointwise lower bound for any pairwise product. Any later use of these two mechanisms must therefore keep their hypotheses and conclusions separate.
\end{remark}
\begin{lemma}\label{lem:1.9}
The sequential map $G:K\to K$ is globally Lipschitz. Let
\[
\kappa:=\frac{M^4}{\varepsilon\lambda_\Omega}.
\]
Then
\begin{equation}\label{eq:LG}
\|G(U)-G(V)\|_X
\le L_G\|U-V\|_X,
\end{equation}
where one may take
\[
L_G
:=
\max\left\{
4\kappa,\,
\kappa(3+4\kappa),\,
\kappa(2+7\kappa+4\kappa^2)
\right\}.
\]
In particular,
\[
L_G
\le
\kappa(4+7\kappa+4\kappa^2).
\]

Moreover, let $\omega'\Subset\omega\Subset\Omega$, set
\[
A:=G(U),\qquad B:=G(V),
\qquad
d_r:=a_r-b_r,
\]
and suppose that, for some $r\in\{1,2,3\}$ and $\delta>0$, the two
coefficients appearing in the $r$-th sequential solve satisfy
\[
q_r^U\ge\delta,
\qquad
q_r^V\ge\delta
\qquad\text{a.e. in }\omega,
\]
where
\[
q_1^U=(u_2u_3)^2,
\qquad
q_1^V=(v_2v_3)^2,
\]
\[
q_2^U
=\frac12u_3^2(u_1^2+a_1^2),
\qquad
q_2^V
=\frac12v_3^2(v_1^2+b_1^2),
\]
and
\[
q_3^U
=
\frac12\bigl((u_1u_2)^2+(a_1a_2)^2\bigr),
\qquad
q_3^V
=
\frac12\bigl((v_1v_2)^2+(b_1b_2)^2\bigr).
\]
Then
\[
\|d_r\|_{L^2(\omega')}
\le
CM
\exp\!\left(
-c\sqrt{\frac{\delta}{\varepsilon}}
\right),
\]
where $C,c>0$ depend only on $d$, $\omega'$, and $\omega$.
\end{lemma}

\begin{proof}
Let
\[
A=G(U),\qquad B=G(V).
\]
For the first sequential solve, the argument used in
Lemma~\ref{lem:1.8} gives
\begin{equation}\label{eq:Gd1}
\|d_1\|_{L^2}
\le
4\kappa\|U-V\|_X.
\end{equation}

For the second component, define
\[
c_U
:=
\frac12u_3^2(u_1^2+a_1^2),
\qquad
c_V
:=
\frac12v_3^2(v_1^2+b_1^2).
\]
Then
\[
-\Delta d_2+\frac{c_U}{\varepsilon}d_2
=
\frac{c_V-c_U}{\varepsilon}b_2.
\]
Since $0\le b_2\le M$, Lemma~\ref{lem:1.4} yields
\[
\|d_2\|_{L^2}
\le
\frac{M}{\varepsilon\lambda_\Omega}
\|c_V-c_U\|_{L^2}.
\]
Using \eqref{eq:ineq3} and adding and subtracting intermediate terms,
\[
\|c_V-c_U\|_{L^2}
\le
M^3
\left(
3\|U-V\|_X+\|d_1\|_{L^2}
\right).
\]
Together with \eqref{eq:Gd1}, this gives
\begin{equation}\label{eq:Gd2}
\|d_2\|_{L^2}
\le
\kappa(3+4\kappa)\|U-V\|_X.
\end{equation}

For the third component, set
\[
e_U
:=
\frac12\bigl((u_1u_2)^2+(a_1a_2)^2\bigr),
\qquad
e_V
:=
\frac12\bigl((v_1v_2)^2+(b_1b_2)^2\bigr).
\]
Then
\[
-\Delta d_3+\frac{e_U}{\varepsilon}d_3
=
\frac{e_V-e_U}{\varepsilon}b_3.
\]
Again $0\le b_3\le M$, and therefore
\[
\|d_3\|_{L^2}
\le
\frac{M}{\varepsilon\lambda_\Omega}
\|e_V-e_U\|_{L^2}.
\]
Now
\[
\begin{aligned}
\|e_V-e_U\|_{L^2}
&\le
\frac12
\|(v_1v_2)^2-(u_1u_2)^2\|_{L^2}
\\
&\qquad
+
\frac12
\|(b_1b_2)^2-(a_1a_2)^2\|_{L^2}
\\
&\le
2M^3\|U-V\|_X
+
M^3
\left(
\|d_1\|_{L^2}
+
\|d_2\|_{L^2}
\right).
\end{aligned}
\]
Using \eqref{eq:Gd1} and \eqref{eq:Gd2}, we obtain
\[
\begin{aligned}
\|d_3\|_{L^2}
&\le
\kappa
\left[
2
+4\kappa
+\kappa(3+4\kappa)
\right]
\|U-V\|_X
\\
&=
\kappa(2+7\kappa+4\kappa^2)
\|U-V\|_X.
\end{aligned}
\]
Taking the maximum of the three componentwise estimates proves
\eqref{eq:LG}.

For the interior estimate, let $q_r^U$ and $q_r^V$ denote the
coefficients in the $r$-th sequential equations defining $a_r$ and
$b_r$.  Under the assumed lower bounds,
Lemma~\ref{lem:A} gives
\[
\|a_r\|_{L^\infty(\omega')}
+
\|b_r\|_{L^\infty(\omega')}
\le
CM
\exp\!\left(
-c\sqrt{\frac{\delta}{\varepsilon}}
\right).
\]
Therefore
\[
\begin{aligned}
\|d_r\|_{L^2(\omega')}
&\le
|\omega'|^{1/2}
\left(
\|a_r\|_{L^\infty(\omega')}
+
\|b_r\|_{L^\infty(\omega')}
\right)
\\
&\le
CM
\exp\!\left(
-c\sqrt{\frac{\delta}{\varepsilon}}
\right),
\end{aligned}
\]
after absorbing $|\omega'|^{1/2}$ into $C$.
\end{proof}
\begin{remark}\label{rem:TG_constants}
The estimates for $T$ and $G$ are worst-case global bounds on the full
invariant set $K$.  In particular,
\[
L_T=4\kappa,
\qquad
L_G\ge4\kappa,
\]
so the present analysis does not imply that the sequential map is a
stronger contraction than the Jacobi/Picard map.  Any numerical advantage
of the sequential update must therefore arise from the actual spatial
coefficients and active-phase geometry rather than from these global
Lipschitz estimates.

Conversely, the deterioration of $L_T$ and $L_G$ as
$\varepsilon\downarrow0$ should not be interpreted as numerical
divergence.  These constants are uniform worst-case estimates on $K$ and
do not exploit the strong local suppression described in
Lemma~\ref{lem:A}.
\end{remark}

\subsection{Convergence Analysis}\label{subsec:convergence}

We now derive the fixed-$\varepsilon$ existence, bracketing, and convergence
properties of the Jacobi/Picard iteration from
Lemmas~\ref{lem:1.3} and \ref{lem:1.5}--\ref{lem:1.9}.
Throughout this subsection, $\varepsilon>0$ is fixed,
$T$ is the Jacobi/Picard map defined in \eqref{eq:8}, and
\[
K:=\{u\in X:0\le u_i\le M\ \text{a.e. in }\Omega,\ i=1,2,3\}.
\]
The unrelaxed iteration is
\[
u^{k+1}=T(u^k),
\]
whereas the relaxed iteration is
\[
u^{k+1}=S_\omega(u^k),
\qquad
S_\omega:=(1-\omega)\operatorname{Id}+\omega T,
\qquad 0<\omega\le1.
\]
Since $K$ is convex and $T(K)\subset K$, one also has
$S_\omega(K)\subset K$.

\begin{theorem}[Existence for every fixed $\varepsilon$]
\label{thm:existence_schauder}
For every $\varepsilon>0$, the map $T$ has at least one fixed point
$u^*\in K$. Equivalently, \eqref{eq:system} admits a weak solution
satisfying
\[
0\le u_i^*\le M
\qquad\text{a.e. in }\Omega,
\quad i=1,2,3.
\]
\end{theorem}

\begin{proof}
The set $K$ is nonempty, closed, bounded, and convex in
$X=(L^2(\Omega))^3$. By Lemma~\ref{lem:1.5}, $T(K)\subset K$;
by Lemma~\ref{lem:1.7}, $T$ is continuous in $X$; and by
Lemma~\ref{lem:1.6}, $T(K)$ is relatively compact in $X$.
Schauder's fixed-point theorem therefore gives $u^*\in K$ with
$T(u^*)=u^*$. Since $T(u^*)\in(H^1(\Omega))^3$, the fixed point belongs
to $(H^1(\Omega))^3$ and satisfies \eqref{eq:system} weakly.
\end{proof}

\begin{remark}\label{rem:schauder_vs_variational}
Theorem~\ref{thm:existence_schauder} concerns fixed points of the
penalized system and does not assert that an arbitrary fixed point is a
global minimizer of $E^\varepsilon$.  Proposition~\ref{prop:penalty_conv},
by contrast, concerns globally minimizing penalized solutions.
\end{remark}

The order-reversing structure of $T$ yields considerably more information
for the unrelaxed iteration.

\begin{theorem}[Monotone bracketing and limiting two-cycle]
\label{thm:sandwich}
Let $\omega=1$ and initialize the Picard iteration with the harmonic
extensions
\[
\Delta u_i^0=0\quad\text{in }\Omega,
\qquad
\operatorname{Tr}u_i^0=\phi_i.
\]
Then the following assertions hold componentwise and a.e. in $\Omega$.

\begin{enumerate}
\item[(i)]
The map $T$ is antitone on $K$:
\[
u\le v\quad\Longrightarrow\quad T(v)\le T(u).
\]

\item[(ii)]
The iterates satisfy
\[
u^1\le u^3\le\cdots\le u^{2k+1}
\le u^{2k+2}\le u^{2k}\le\cdots\le u^2\le u^0.
\]
Hence there exist $\underline u,\overline u\in K$ such that
\[
u^{2k+1}\uparrow\underline u,
\qquad
u^{2k}\downarrow\overline u
\]
pointwise a.e. and strongly in $X$, and
\[
\underline u\le\overline u,
\qquad
T(\overline u)=\underline u,
\qquad
T(\underline u)=\overline u.
\]

\item[(iii)]
Every fixed point $u^*\in K$ satisfies
\[
\underline u\le u^*\le\overline u.
\]

\item[(iv)]
The full sequence $(u^k)$ converges in $X$ if and only if
$\underline u=\overline u$.  In that case, the common limit is the
unique fixed point of $T$ in $K$.

\item[(v)]
For every $k\ge0$,
\[
0\le\overline u-\underline u
\le |u^{k+1}-u^k|
\qquad\text{a.e. in }\Omega,
\]
and therefore
\begin{equation}\label{eq:bracketing_gap}
\|\overline u-\underline u\|_X
\le
\|u^{k+1}-u^k\|_X.
\end{equation}
\end{enumerate}
\end{theorem}

\begin{proof}
If $u\le v$, then for each $i$,
\[
q_i(u):=\left(\prod_{j\ne i}u_j\right)^2
\le
\left(\prod_{j\ne i}v_j\right)^2=:q_i(v).
\]
The comparison principle in Lemma~\ref{lem:1.3} therefore gives
$T(v)_i\le T(u)_i$, proving (i).

The harmonic extension corresponds to zero reaction coefficient, so
$u^1=T(u^0)\le u^0$. Antitonicity gives $u^1\le u^2$, while comparison
with the harmonic extension gives $u^2\le u^0$. Repeated application of
$T$ yields the interlaced chain in (ii). The odd subsequence is increasing
and the even subsequence decreasing, while all iterates lie in $K$.
Hence both converge pointwise a.e.; dominated convergence gives strong
convergence in $X$. Continuity of $T$ then yields
\[
T(\overline u)=\underline u,
\qquad
T(\underline u)=\overline u.
\]

If $u^*=T(u^*)$, comparison with $u^0$ and induction give
\[
u^{2k+1}\le u^*\le u^{2k}
\qquad(k\ge0),
\]
and (iii) follows by passage to the limit. Assertion (iv) follows
immediately from this bracketing.

Finally, if $k=2r$ then
\[
u^{2r+1}\le\underline u\le\overline u\le u^{2r},
\]
whereas if $k=2r+1$ then
\[
u^{2r+1}\le\underline u\le\overline u\le u^{2r+2}.
\]
Thus
\[
0\le\overline u-\underline u
\le|u^{k+1}-u^k|,
\]
which proves (v).
\end{proof}

The global Lipschitz estimates give geometric convergence under explicit
sufficient conditions.

\begin{theorem}[Geometric convergence under contraction conditions]
\label{thm:contraction}
Let
\[
L_T:=\frac{4M^4}{\varepsilon\lambda_\Omega}.
\]

\begin{enumerate}
\item[(i)]
If
\begin{equation}\label{eq:contraction_condition}
L_T<1,
\end{equation}
then $T$ has a unique fixed point $u^*\in K$. For every
$\omega\in(0,1]$, the relaxed Picard iteration satisfies
\[
\|u^k-u^*\|_X
\le
\bigl(1-\omega(1-L_T)\bigr)^k
\|u^0-u^*\|_X.
\]
Among these worst-case bounds, the convergence factor is smallest for
$\omega=1$.

\item[(ii)]
For the sequential map $G$, set
\[
\kappa:=\frac{M^4}{\varepsilon\lambda_\Omega}
\]
and
\[
L_G:=
\max\left\{
4\kappa,\,
\kappa(3+4\kappa),\,
\kappa(2+7\kappa+4\kappa^2)
\right\}.
\]
If
\begin{equation}\label{eq:G_contraction_condition}
L_G<1,
\end{equation}
then $G$ has a unique fixed point in $K$, and
$u^{k+1}=G(u^k)$ converges geometrically to it for every $u^0\in K$.

\item[(iii)]
The fixed-point sets of $T$ and $G$ coincide and are precisely the weak
solutions of \eqref{eq:system} belonging to $K$. Hence either contraction
condition implies uniqueness of the weak solution in $K$.
\end{enumerate}
\end{theorem}

\begin{proof}
Under \eqref{eq:contraction_condition},
Lemma~\ref{lem:1.8} and Banach's fixed-point theorem give a unique
fixed point of $T$. Moreover,
\[
\begin{aligned}
\|S_\omega(u)-S_\omega(v)\|_X
&\le
(1-\omega)\|u-v\|_X
+\omega\|T(u)-T(v)\|_X\\
&\le
\bigl(1-\omega(1-L_T)\bigr)\|u-v\|_X,
\end{aligned}
\]
which gives the stated geometric estimate.

Under \eqref{eq:G_contraction_condition}, the corresponding assertion
follows from Lemma~\ref{lem:1.9} and Banach's theorem.

A fixed point of $T$ satisfies \eqref{eq:system} by definition.
Conversely, uniqueness of each linear Dirichlet subproblem implies that
every solution of \eqref{eq:system} in $K$ is a fixed point of $T$.
At a fixed point of $G$, the averaged coefficients reduce to
\[
\frac12u_3^2(u_1^2+u_1^2)=u_1^2u_3^2,
\qquad
\frac12\bigl((u_1u_2)^2+(u_1u_2)^2\bigr)=(u_1u_2)^2,
\]
so the same conclusion holds for $G$.
\end{proof}

\begin{remark}\label{rem:fixed_eps_vs_limit_nonunique}
The uniqueness in Theorem~\ref{thm:contraction} is a fixed-$\varepsilon$
result under a sufficient contraction condition. It does not imply
uniqueness of the limiting constrained minimization problem; indeed,
Example~\ref{ex:nonuniqueness} shows that the latter may have distinct
global minimizers.
\end{remark}

We finally record the local suppression produced by a uniformly positive
competition coefficient.

\begin{theorem}[Interior exponential smallness of the iterates]
\label{thm:interior}
Let
\[
\omega'\Subset\omega\Subset\Omega
\]
and fix $i\in\{1,2,3\}$. Assume that, for some $k_0\ge0$ and $\delta>0$,
\[
q_i(u^k):=
\left(\prod_{j\ne i}u_j^k\right)^2
\ge\delta
\qquad\text{a.e. in }\omega,
\quad k\ge k_0.
\]
Then there exist constants $C,c>0$, independent of $k$ and
$\varepsilon$, such that the unrelaxed iteration satisfies
\[
\|u_i^{k+1}\|_{L^\infty(\omega')}
\le
CM\exp\!\left(-c\sqrt{\frac{\delta}{\varepsilon}}\right),
\qquad k\ge k_0.
\]
Consequently, for $k,l\ge k_0+1$,
\[
\|u_i^k-u_i^l\|_{L^2(\omega')}
\le
2CM|\omega'|^{1/2}
\exp\!\left(-c\sqrt{\frac{\delta}{\varepsilon}}\right).
\]

For $0<\omega<1$, the relaxed iteration satisfies
\[
\|u_i^k\|_{L^\infty(\omega')}
\le
CM\exp\!\left(-c\sqrt{\frac{\delta}{\varepsilon}}\right)
+(1-\omega)^{k-k_0}M,
\qquad k\ge k_0.
\]
\end{theorem}

\begin{proof}
For $k\ge k_0$, the unrelaxed update
$\widehat u_i^{k+1}=T(u^k)_i$ solves
\[
-\Delta\widehat u_i^{k+1}
+\frac{q_i(u^k)}{\varepsilon}\widehat u_i^{k+1}=0
\]
with trace $\phi_i$. Lemma~\ref{lem:A} gives the first estimate uniformly
in $k$. The difference estimate follows from the triangle inequality.

For $0<\omega<1$, set
\[
A_\varepsilon
:=
CM\exp\!\left(-c\sqrt{\frac{\delta}{\varepsilon}}\right).
\]
Since
\[
u_i^{k+1}
=(1-\omega)u_i^k+\omega\widehat u_i^{k+1},
\qquad
\|\widehat u_i^{k+1}\|_{L^\infty(\omega')}
\le A_\varepsilon,
\]
iteration gives
\[
\|u_i^k\|_{L^\infty(\omega')}
\le
(1-\omega)^{k-k_0}
\|u_i^{k_0}\|_{L^\infty(\omega')}
+
\bigl(1-(1-\omega)^{k-k_0}\bigr)A_\varepsilon.
\]
The conclusion follows from $\|u_i^{k_0}\|_{L^\infty}\le M$.
\end{proof}

Theorem~\ref{thm:interior} is an absolute local smallness result; it does
not provide a global convergence rate or an estimate of the discretization
error. Its lower-bound hypothesis can, however, be checked on computed
iterates and identifies regions where the corresponding component is
strongly suppressed.

\begin{remark}[Scope of the convergence theory]
\label{rem:no_global}
Condition \eqref{eq:contraction_condition} is sufficient, not necessary,
and requires
\[
\varepsilon>\frac{4M^4}{\lambda_\Omega}.
\]
It therefore does not cover the strong-competition regime
$\varepsilon\downarrow0$. The deterioration of the global Lipschitz bound
as $\varepsilon$ decreases does not imply divergence of the iteration; it
only shows that the present Banach-contraction argument is not uniform in
$\varepsilon$.

For fixed $\varepsilon$, the preceding results provide existence,
monotone bracketing of all fixed points, and geometric convergence under
explicit contraction conditions. Without contraction, however, the even
and odd Picard subsequences may in principle approach a nontrivial
two-cycle.

This fixed-$\varepsilon$ issue is distinct from nonuniqueness of the
limiting constrained problem.  Proposition~\ref{prop:penalty_conv}
describes convergence of global penalized minimizers toward the set of
constrained global minimizers, while Example~\ref{ex:nonuniqueness} shows
that this set need not be a singleton.

In the numerical experiments, continuation in $\varepsilon$ is therefore
used as a warm-start and branch-selection heuristic. Its schedule and
stopping tolerances are algorithmic parameters. For the unrelaxed
iteration, \eqref{eq:bracketing_gap} nevertheless gives a rigorous
a posteriori interpretation of the successive-difference stopping
criterion.
\end{remark}

\section{Projected Gradient Method}\label{sec:pg}
The second scheme alternates a gradient step for the Dirichlet energy with a pointwise Euclidean projection onto the nonconvex ternary set. Because the projection is multivalued at ties, the implementation uses an explicit measurable tie-breaking rule. The methods are intended to compute candidate constrained configurations; no global convergence claim is made for arbitrary initial data in the nonconvex projected-gradient scheme.

Let
\[
\Sigma
:=
\{a=(a_1,a_2,a_3)\in\mathbb{R}^3:
a_i\ge0,\ a_1a_2a_3=0\}
=
\Sigma_1\cup\Sigma_2\cup\Sigma_3,
\]
where
\[
\Sigma_r
:=
\{a\in\mathbb{R}^3:
a_i\ge0\ \text{for all }i,\ a_r=0\}.
\]
We first record the metric projection in $L^2$. This result concerns only
the pointwise segregation constraint; preservation of traces and $H^1$
regularity is addressed separately through the discrete formulation below.

For $v=(v_1,v_2,v_3)\in L^2(\Omega)^3$, define $\mathcal Pv$ pointwise.
Let
\[
k^*(x)
:=
\min\operatorname*{argmin}_{r\in\{1,2,3\}}(v_r(x))^+,
\]
where $s^+:=\max\{s,0\}$ and the minimum provides deterministic
tie-breaking. Set
\[
(\mathcal Pv)_i(x)
=
\begin{cases}
(v_i(x))^+,&i\ne k^*(x),\\[1mm]
0,&i=k^*(x).
\end{cases}
\]

\begin{lemma}\label{lem:projection_corrected}
Let
\[
\mathfrak C
:=
\{w\in L^2(\Omega)^3:
w(x)\in\Sigma\ \text{for a.e. }x\in\Omega\}.
\]
Then $\mathcal Pv\in\mathfrak C$ and
\[
\|v-\mathcal Pv\|_{L^2(\Omega)^3}^2
=
\min_{w\in\mathfrak C}
\|v-w\|_{L^2(\Omega)^3}^2.
\]
Thus $\mathcal P$ is a single-valued measurable selection of the generally
set-valued metric projection onto $\mathfrak C$.
\end{lemma}

\begin{proof}
For a fixed $x$ and $r\in\{1,2,3\}$, the Euclidean projection of $v(x)$
onto the closed convex face $\Sigma_r$ is
\[
\pi_r(v(x))_i
=
\begin{cases}
0,&i=r,\\
(v_i(x))^+,&i\ne r.
\end{cases}
\]
Its squared distance is
\[
\begin{aligned}
|v(x)-\pi_r(v(x))|^2
&=
|v_r(x)|^2
+
\sum_{i\ne r}|v_i(x)-(v_i(x))^+|^2\\
&=
((v_r(x))^+)^2
+
\sum_{i=1}^3((v_i(x))^-)^2,
\end{aligned}
\]
where $s^-:=\max\{-s,0\}$. The second term is independent of $r$;
therefore, a nearest face is obtained by minimizing $(v_r(x))^+$.

Since the functions $(v_r)^+$ are measurable, the deterministic rule
\[
k^*(x)
=
\min\operatorname*{argmin}_{r\in\{1,2,3\}}(v_r(x))^+
\]
defines a measurable index function. Hence $\mathcal Pv$ is measurable.
Moreover,
\[
|(\mathcal Pv)_i|\le |v_i|
\qquad\text{a.e. in }\Omega,
\]
so $\mathcal Pv\in L^2(\Omega)^3$. Pointwise optimality gives
$\mathcal Pv\in\mathfrak C$, and integration of the pointwise minimum
proves the assertion.
\end{proof}

\begin{remark}[Continuous versus discrete projection]
\label{rem:projection_scope}
Changing values on $\partial\Omega$ does not affect an
$L^2(\Omega)$ equivalence class, so reimposing boundary values after an
$L^2$ projection has no intrinsic continuous meaning. Moreover, projection
onto the union of faces may create switching discontinuities and need not
map $(H^1(\Omega))^3$ into itself. Accordingly,
Algorithm~\ref{alg:projected_gradient} is formulated as a
finite-dimensional nodal method with the Dirichlet degrees of freedom held
fixed.
\end{remark}

Let $z_i\in\mathbb{R}^{N_I}$ denote the interior nodal values of component
$i$. For the finite-difference discretization, let $A$ be the symmetric
positive definite Dirichlet matrix for $-\Delta_h$, and let $b_i$ contain
the contribution of the fixed boundary data. Up to an irrelevant additive
constant and a fixed positive mesh-dependent scaling, the discrete
Dirichlet energy can be written as
\begin{equation}\label{eq:discrete_energy_pg}
F_h(z)
:=
\frac12\sum_{i=1}^3 z_i^\top A z_i
+
\sum_{i=1}^3 b_i^\top z_i,
\qquad
\nabla_iF_h(z)=Az_i+b_i.
\end{equation}

Define
\[
\mathfrak C_h
:=
\left\{
z\in(\mathbb{R}^{N_I})^3:
(z_1(r),z_2(r),z_3(r))\in\Sigma
\text{ for every interior node }r
\right\}.
\]
The Euclidean projection onto $\mathfrak C_h$ is obtained by applying the
pointwise formula of Lemma~\ref{lem:projection_corrected} independently at
each interior node.

\begin{remark}[Discrete nature of the nodal constraint]
\label{rem:nodal_constraint}
The set $\mathfrak C_h$ enforces partial segregation at the discrete
degrees of freedom. For the finite-difference formulation this is the
natural discrete analogue of the continuous pointwise constraint.

If the same nodal vectors are interpreted as continuous piecewise-affine
finite-element functions, nodal segregation alone does not imply
\[
u_1u_2u_3=0
\]
throughout each element. Hence $\mathfrak C_h$ should not be regarded as a
conforming finite-element subspace of the continuous admissible set $S$.
The projected-gradient method below is therefore understood as a
finite-dimensional nodal optimization scheme.
\end{remark}

\begin{proposition}[Descent of the exact projected-gradient step]
\label{prop:pg_descent}
Let
\[
L_h:=|\lambda_{\max}(A)|,
\]
let $z^k\in\mathfrak C_h$, and choose
\[
0<\tau<\frac1{L_h}.
\]
Suppose
\[
z^{k+1}\in
\operatorname*{argmin}_{z\in\mathfrak C_h}
\left\|
z-\bigl(z^k-\tau\nabla F_h(z^k)\bigr)
\right\|_2^2.
\]
Then
\[
F_h(z^{k+1})
\le
F_h(z^k)
-
\left(
\frac1{2\tau}-\frac{L_h}{2}
\right)
\|z^{k+1}-z^k\|_2^2.
\]
Consequently, the energy is nonincreasing, the iterates remain in a
bounded level set, and
\[
\|z^{k+1}-z^k\|_2\longrightarrow0.
\]
Every accumulation point $\bar z$ satisfies the projected fixed-point
condition
\[
\bar z\in
\operatorname{Proj}_{\mathfrak C_h}
\bigl(\bar z-\tau\nabla F_h(\bar z)\bigr),
\]
where the projection denotes the full set-valued metric projection.
No global-minimality conclusion follows from this condition.
\end{proposition}

\begin{proof}
Since $z^k\in\mathfrak C_h$ is an admissible comparison point,
\[
\left\|
z^{k+1}-z^k+\tau\nabla F_h(z^k)
\right\|_2^2
\le
\tau^2\|\nabla F_h(z^k)\|_2^2.
\]
After expansion, we get
$$
\|z^{k+1}-z^k\|_2^2 + 2\tau\langle \nabla F_h(z^k), z^{k+1}-z^k\rangle + \tau^2\|\nabla F_h(z^k)\|_2^2\le \tau^2\|\nabla F_h(z^k)\|_2^2,
$$
which in turn implies
\begin{equation}\label{descent-expansion}
    \langle
\nabla F_h(z^k),z^{k+1}-z^k
\rangle
\le
-\frac1{2\tau}\|z^{k+1}-z^k\|_2^2.
\end{equation}
Note that due to $A$ is symmetric, then for every $\omega_1,\omega_2\in\mathfrak C_h$ we have
$$
\|\nabla F_h(\omega_1)-\nabla F_h(\omega_2)\|_2\le \|A\|_2\cdot\|\omega_1-\omega_2\|_2= |\lambda_{\max}(A)|\cdot\|\omega_1-\omega_2\|_2.
$$
Hence $\nabla F_h$ is $L_h$-Lipschitz, and the standard descent lemma provides the quadratic upper bound
\[
F_h(z^{k+1})
\le
F_h(z^k)
+
\langle \nabla F_h(z^k), z^{k+1}-z^k \rangle
+
\frac{L_h}{2} \|z^{k+1}-z^k\|_2^2.
\]
Substituting  \eqref{descent-expansion} directly into this inequality yields
\[
F_h(z^{k+1})
\le
F_h(z^k)
-
\frac1{2\tau}\|z^{k+1}-z^k\|_2^2
+
\frac{L_h}{2}\|z^{k+1}-z^k\|_2^2.
\]
Thus,
\[
F_h(z^{k+1})
\le
F_h(z^k)
-
\left(
\frac1{2\tau}-\frac{L_h}{2}
\right)
\|z^{k+1}-z^k\|_2^2.
\]
Because $A$ is positive definite on the interior variables, the level sets
of $F_h$ are bounded. Summing the descent estimate yields
\[
\sum_{k=0}^\infty
\|z^{k+1}-z^k\|_2^2<\infty,
\]
and therefore
\[
\|z^{k+1}-z^k\|_2\to0.
\]

If $z^{k_j}\to\bar z$, then
$z^{k_j+1}-z^{k_j}\to0$ implies
$z^{k_j+1}\to\bar z$. Since
\[
z^{k_j+1}\in
\operatorname{Proj}_{\mathfrak C_h}
\bigl(
z^{k_j}-\tau\nabla F_h(z^{k_j})
\bigr)
\]
and the graph of the metric projection onto a closed subset of a
finite-dimensional Euclidean space is closed, passage to the limit gives
the stated projected fixed-point condition.

\end{proof}

\begin{algorithm}[H]
\caption{Discrete projected-gradient method}
\label{alg:projected_gradient}
\begin{algorithmic}[1]
\Require Fixed boundary vectors, a feasible interior vector
$z^0\in\mathfrak C_h$, tolerance $\mathrm{tol}>0$, and
$0<\tau<1/\lambda_{\max}(A)$.
\For{$k=0,1,2,\ldots$ until convergence}
    \For{$i=1,2,3$}
        \State $g_i^k\gets Az_i^k+b_i$.
        \State $\widetilde z_i^{k+1}\gets z_i^k-\tau g_i^k$.
    \EndFor
    \State Apply the pointwise projection of
    Lemma~\ref{lem:projection_corrected} at every interior node:
    \[
    z^{k+1}\gets\mathcal P_h(\widetilde z^{k+1}).
    \]
    \If{$\|z^{k+1}-z^k\|_2\le\mathrm{tol}$}
        \State \textbf{break}.
    \EndIf
\EndFor
\State \Return $z^{k+1}$ together with the fixed Dirichlet values.
\end{algorithmic}
\end{algorithm}

The harmonic extension generally need not belong to $\mathfrak C_h$.
It is therefore first computed with the prescribed Dirichlet values and
then projected only at the interior nodes to obtain a feasible initial
iterate $z^0\in\mathfrak C_h$. Boundary nodes are never included in the
projection.

\begin{remark}[Step-size selection]\label{rem:stepsize}
For projection onto the nonconvex set $\mathfrak C_h$, the descent
argument in Proposition~\ref{prop:pg_descent} requires
\[
\tau<\frac1{|\lambda_{\max}(A)|}.
\]
The less restrictive condition
$\tau<2/|\lambda_{\max}(A)|$ is the classical unconstrained
gradient-descent bound and is not a general guarantee for the present
nonconvex projection.

For the standard five-point Dirichlet Laplacian on a uniform
two-dimensional grid,
\[
|\lambda_{\max}(A)|<8h^{-2},
\]
so
\[
\tau\le\frac{h^2}{8}
\]
is a conservative admissible choice; in particular,
$\tau=0.1h^2$ satisfies the sufficient condition.

For a mass-lumped $P_1$ finite-element discretization, the gradient in
the lumped $L^2$ metric is
\[
M_L^{-1}(K_{II}z_i+K_{IB}g_i),
\]
and the corresponding step-size restriction involves the largest
generalized eigenvalue of $(K_{II},M_{L,II})$. Since the same positive
nodal mass multiplies all three component coordinates at a given node,
the weighted metric projection retains the pointwise formula of
Lemma~\ref{lem:projection_corrected}.
\end{remark}

An inertial extrapolation can reduce iteration counts in practice, but
the classical FISTA theory applies to convex composite problems
\cite{BeckTeboulle2009FISTA} and does not directly apply to the indicator
of $\mathfrak C_h$. We therefore use the following safeguarded inertial
variant only as a numerical heuristic; compare to
\cite{OchsPock2019}.

The acceptance test makes the accepted energy values nonincreasing by
construction, but it does not turn the scheme into classical convex
FISTA. After a restart, $y^k=z^k$. With the optional bias,
\[
\frac{z^k-\tau\nabla F_h(z^k)+\eta z^k}{1+\eta}
=
z^k-\frac{\tau}{1+\eta}\nabla F_h(z^k),
\]
so the restarted trial is an exact projected-gradient step with effective
step size $\tau/(1+\eta)$ whenever the exact selector is used. Once this
effective step is below $1/L_h$,
Proposition~\ref{prop:pg_descent} guarantees descent.

A near-tie phase-selection rule may be useful as a floating-point
hysteresis convention. We denote this optional feasible selector by
$\widetilde{\mathcal P}_h$. It always returns a point in
$\mathfrak C_h$, but it may deliberately select a non-nearest face and
therefore need not be a metric projection. Consequently,
Proposition~\ref{prop:pg_descent} does not apply to such a trial.

To retain a rigorous descent safeguard, after $N_{\rm bt}$ rejected
trials Algorithm~\ref{alg:fista} switches to the exact metric-projection
selector $\mathcal P_h$ and continues reducing the step size whenever
necessary. Thus $N_{\rm bt}$ is a switching threshold, rather than a
maximum number of backtracking reductions. This exact fallback is the
mechanism used in Proposition~\ref{prop:fista_descent} to guarantee
finite termination of the inner loop.

\begin{proposition}[Descent of the safeguarded inertial method]
\label{prop:fista_descent}
Let $L_h:=|\lambda_{\max}(A)|$, and assume that every trial selector
takes values in $\mathfrak C_h$. After at most
$N_{\mathrm{bt}}$ rejected trials, suppose that the algorithm
uses the exact metric projection onto $\mathfrak C_h$ while
continuing to reduce the step size by the factor $\rho\in(0,1)$
whenever a trial is rejected. Then, for every initialization
$z^0\in\mathfrak C_h$:

\begin{enumerate}
\item[(i)]
At every iteration $k$, the \texttt{Repeat--Until} loop terminates
after finitely many backtracking steps.

\item[(ii)]
Every accepted step satisfies
\[
F_h(z^{k+1})\le F_h(z^k).
\]
Hence all iterates remain in the compact sublevel set
\[
\{z\in\mathfrak C_h:F_h(z)\le F_h(z^0)\}.
\]
In particular, if infinitely many iterates are generated, the
sequence $\{z^k\}$ has at least one accumulation point.
\end{enumerate}
\end{proposition}

\begin{proof}
Suppose that, at some iteration $k$, the inner loop does not
terminate. Every trial is then rejected, so after the first
rejection the algorithm restarts with $y^k=z^k$ and repeatedly
replaces $\tau$ by $\rho\tau$. After
$N_{\mathrm{bt}}$ rejections it uses the exact metric
projection.

With $y^k=z^k$, the biased trial argument is
\[
\frac{z^k-\tau\nabla F_h(z^k)+\eta z^k}{1+\eta}
=
z^k-\tau'\nabla F_h(z^k),
\qquad
\tau':=\frac{\tau}{1+\eta}.
\]
Since $0<\rho<1$, repeated backtracking gives $\tau'\to0$.
Thus, after finitely many additional reductions,
\[
\tau'<\frac1{L_h}.
\]
At this stage,
\[
z^{\rm trial}
\in
\operatorname{Proj}_{\mathfrak C_h}
\bigl(z^k-\tau'\nabla F_h(z^k)\bigr).
\]
Because $z^k\in\mathfrak C_h$, Proposition~\ref{prop:pg_descent}
yields
\[
F_h(z^{\rm trial})
\le
F_h(z^k)
-
\left(
\frac1{2\tau'}-\frac{L_h}{2}
\right)
\|z^{\rm trial}-z^k\|_2^2
\le F_h(z^k).
\]
Hence the trial satisfies the acceptance condition, contradicting
nontermination. This proves (i).

Assertion (ii) follows directly from the acceptance criterion.
Since $A$ is positive definite, $F_h$ is coercive. Moreover,
$\mathfrak C_h$ is closed; hence
\[
\{z\in\mathfrak C_h:F_h(z)\le F_h(z^0)\}
\]
is compact. Thus all accepted iterates remain in this set, and any
infinite sequence of iterates possesses an accumulation point.
\end{proof}

Let $\mathcal P_h$ denote the exact metric-projection selector introduced above.
If a near-tie phase-selection or hysteresis rule is used, denote the corresponding
feasible selector by $\widetilde{\mathcal P}_h$, so that
\[
\widetilde{\mathcal P}_h(v)\in\mathfrak C_h
\qquad\text{for every }v,
\]
although $\widetilde{\mathcal P}_h(v)$ need not be a nearest point to $v$.
If no modified selector is used, we simply set
$\widetilde{\mathcal P}_h=\mathcal P_h$.
The integer $N_{\rm bt}\ge0$ is the number of rejected trials after which
the algorithm switches to the exact selector $\mathcal P_h$; it is not
a maximum number of backtracking reductions.

\begin{algorithm}[htbp]
\caption{Safeguarded inertial projected-gradient method}
\label{alg:fista}
\begin{algorithmic}[1]
\Require A feasible $z^0\in\mathfrak C_h$, $z^{-1}=z^0$, $t_0=1$, $K_{\max},$
initial step $\tau_0>0$, shrink factor $\rho\in(0,1)$,
optional bias $\eta\ge0$, tolerance $\mathrm{tol}>0$,
and switching threshold $N_{\rm bt}\ge0$.
\For{$k=0,1,2,\ldots,K_{\max}-1$}
    \State
    $\widehat t_{k+1}\gets(1+\sqrt{1+4t_k^2})/2$,
    $\beta_k\gets(t_k-1)/\widehat t_{k+1}$.
    \State
    $y^k\gets z^k+\beta_k(z^k-z^{k-1})$,
    $\tau\gets\tau_k$,
    $\mathrm{restart}\gets\mathrm{false}$,
    $n_{\rm bt}\gets0$.
    \Repeat
        \State $r\gets y^k-\tau\nabla F_h(y^k)$.
        \If{$\eta>0$}
            \State $r\gets(r+\eta z^k)/(1+\eta)$.
        \EndIf
        \If{$n_{\rm bt}<N_{\rm bt}$}
            \State
            $z^{\mathrm{trial}}\gets\widetilde{\mathcal P}_h(r)$.
        \Else
            \State
            $z^{\mathrm{trial}}\gets\mathcal P_h(r)$.
        \EndIf
        \If{$F_h(z^{\mathrm{trial}})>F_h(z^k)$}
            \State
            $y^k\gets z^k$,
            $\mathrm{restart}\gets\mathrm{true}$,
            $\tau\gets\rho\tau$,
            $n_{\rm bt}\gets n_{\rm bt}+1$.
        \EndIf
    \Until{$F_h(z^{\mathrm{trial}})\le F_h(z^k)$}
    \State
    $z^{k+1}\gets z^{\mathrm{trial}}$,
    $\tau_{k+1}\gets\tau$.
    \If{$\mathrm{restart}=\mathrm{true}$}
        \State $t_{k+1}\gets1$.
    \Else
        \State $t_{k+1}\gets\widehat t_{k+1}$.
    \EndIf
    \If{$\|z^{k+1}-z^k\|_2\le\mathrm{tol}$}
        \State \textbf{break}.
    \EndIf
\EndFor
\State \Return $z^{k+1}$ together with the fixed Dirichlet values.
\end{algorithmic}
\end{algorithm}

\begin{remark}[Proximal interpretation]
\label{rem:prox_interpretation}
The metric projection onto $\mathfrak C_h$ is equivalently the
set-valued proximal mapping of its indicator:
\[
\operatorname{Prox}_{\tau I_{\mathfrak C_h}}(v)
=
\operatorname{Proj}_{\mathfrak C_h}(v).
\]
The selector $\mathcal P_h$ chooses one element of this set. Since
$\mathfrak C_h$ is nonconvex, this interpretation does not invoke the
standard convex forward--backward or FISTA convergence theory.
\end{remark}

\begin{remark}\label{rem:pg_scope}
Algorithm~\ref{alg:projected_gradient} has the finite-dimensional descent
property of Proposition~\ref{prop:pg_descent}, and every accumulation
point satisfies the corresponding projected fixed-point condition.
Neither statement implies global optimality for the nonconvex problem.

The safeguarded inertial method is used only as a numerical acceleration
heuristic. Its behavior may depend on initialization, restart,
phase-selection, and bias parameters, and no FISTA-type convergence rate
or global-minimization property is claimed.
\end{remark}

\section{Simulations}\label{sec: simulations}
 
This section reports simulations for the two algorithms of
Sections~\ref{sec: Penal_algorithms} and \ref{sec:pg}.  We first fix the discrete setting
and the reported quantities, so that the two methods are compared on the same
footing and every number is reproducible.
 
\smallskip\noindent\emph{Grids and energy.}
In two dimensions $\Omega=[-1,1]^2$ carries a uniform grid of spacing $h$; in
one dimension $\Omega=(0,1)$ is divided into $N$ subintervals $h=1/N$.  Both  algorithms are always run on the same grid for the same data.  We report the following   energies  
\begin{equation}\label{eq:discrete_energy}
  E_h(U):=\sum_{i=1}^3\sum_{e\in\mathcal E_h}\bigl(u_i(e^+)-u_i(e^-)\bigr)^2
  \quad (d=2),\\
  \end{equation}
  \begin{equation}\label{eq:discrete_energy1}
  E_h(U):=\frac1h\sum_{i=1}^3\sum_j\bigl(u_i(x_{j+1})-u_i(x_j)\bigr)^2
  \quad (d=1),
\end{equation}
$\mathcal E_h$ being the set of grid edges; it approximates
$E(U)=\sum_i\int_\Omega|\nabla u_i|^2$ \emph{without} a factor $\frac12$.  The
quadratic form $F_h$ of (32) equals $E_h/2$ up to an additive constant and is
never reported.  The penalty term is
$P_h^\varepsilon(U):=(h^d/\varepsilon)\sum_j w_j(u_1u_2u_3)^2(x_j)$ with
trapezoidal weights, and $E_h^\varepsilon:=E_h+P_h^\varepsilon$.  Since the
discrete harmonic extension minimizes $E_h$ among nodal functions with the
prescribed data, every computed configuration satisfies
\begin{equation}\label{eq:harmonic_lower_bound}
  E_h(U)\ \ge\ E_h(U_{\mathrm{harm}})=:E_h^{\mathrm{harm}},
\end{equation}
which we report throughout as a strict check on the implementation and as a
measure of what the constraint costs.
 
\smallskip\noindent\emph{Parameters.}
The penalized system is solved by the relaxed Jacobi/Picard iteration of
Algorithm~\ref{alg:monotone_iterative} with $\omega=0.5$, with continuation in $\varepsilon$ from
$10^{-3}$ down to the target value by factors of $10$; each level stops when
the successive difference falls below $10^{-8}$ in the nodal maximum norm (two
dimensions) or $10^{-10}$ in $\|U\|_X=\max_i\|u_i\|_{L^2}$ (one dimension).
The linear systems are solved by preconditioned conjugate gradients
warm-started at the previous iterate, and the outer iteration is accelerated by
Anderson mixing of the residual $S_\omega(U)-U$ with a restart safeguard and
clipping to $[0,M]$.  Acceleration reaches the same fixed point at a fraction of the
cost: $142$ sweeps instead of about $430$ for the configuration of
Example~\ref{ex:numerical_square} at $h=0.01$.  The
projected-gradient computations use the \emph{exact} metric projection $P_h$ of
Lemma~\ref{lem:projection_corrected} as the only selector, with no near-tie
rule and no proximal bias, so that the safeguarded inertial method of
Algorithm~\ref{alg:fista} differs from Algorithm~\ref{alg:projected_gradient} only by the
energy-acceptance test; the initial step is $\tau_0=0.03h^2$, the shrink factor
$\rho=0.5$, the stopping tolerance $10^{-11}$.  Since
$|\lambda_{\max}(A)|<8h^{-2}$ in two and $<4h^{-2}$ in one dimension,
Proposition~\ref{prop:pg_descent} applies throughout; in one dimension
$\tau=0.2h^2$ gives $\tau|\lambda_{\max}(A)|=0.80$.  Because the projection
sets one nodal component to the floating-point zero, the nodal product vanishes
identically: $\max_r|u_1u_2u_3|(r)=0$ \emph{exactly} at every iterate of every
run, and no constraint-violation diagnostic is displayed.  All energies were
reproduced on a second machine to five digits, except the fourth digit of
$E_h^{\mathrm{pg}}$: where two components are exactly equal the projection may
assign either of them to the vanishing face.
 
\smallskip\noindent\emph{Threshold interfaces.}
At fixed $\varepsilon$ a nontrivial nonnegative solution of the linear
component equation is strictly positive in the interior, so the penalised
functions possess no interior free boundary $\partial\{u_i^\varepsilon>0\}$.
All plotted curves are therefore, for both methods,
\begin{equation}\label{eq:threshold}
  \Gamma_{i,h}^{\varepsilon,\delta}
    :=\partial\{x\in\Omega_h:\ u_{i,h}^\varepsilon(x)>\delta\},\qquad \delta=10^{-5}.
\end{equation}
The value $\delta\approx\sqrt\varepsilon$ is a convention, not a consequence of
the estimates of Section~\ref{sec: Penal_algorithms}: Lemma~3.12 and
Theorem~3.26 give a decay length $\sqrt{\varepsilon/q_0}$ only where the
competing coefficient $q_i:=(u_ju_\ell)^2$ has a positive lower bound $q_0$,
and $\sqrt\varepsilon\approx3\cdot10^{-5}\ll h$, so the layer itself is not
resolved.  Where $q_i$ degenerates --- wherever a second component has already
vanished, so that the ternary constraint is met pairwise and $u_i$ is merely
small rather than exponentially small --- the level set is not a layer at all
and its position is set by $\delta$ rather than by the geometry.  We call a
point of $\Gamma_{i,h}^{\varepsilon,\delta}$ \emph{resolved} when
\begin{equation}\label{eq:resolved}
  \sqrt{\varepsilon/q_i}\ \le\ h ,
\end{equation}
draw the resolved part solid and the remainder dotted, and read only the former
as an interface.  For BC1 at $h=0.01$ the resolved fractions are $54\%$,
$54\%$ and $71\%$ of contour length, and between $\varepsilon=10^{-7}$ and
$10^{-9}$ the level sets move by a Hausdorff distance of $0.12$--$0.14$,
essentially all of it in the dotted part; this is not incomplete convergence,
since continuing to a residual of $4\cdot10^{-13}$ moves no node across the
level $\delta$.  The distinction concerns the penalised layer only, a
projected-gradient level set bounding an exact nodal zero set.
 
\subsection{Nondegenerate Instances}
 \begin{example}\label{ex:numerical_square}
Let $\Omega=[-1,1]^2$ and take
\[
  \phi_1=(-y)^+,\qquad \phi_2=y^+,\qquad \phi_3=c >0\qquad\text{on }\partial\Omega,
\]
which satisfies $\phi_1\phi_2\phi_3=0$ on $\partial\Omega$ identically, since
$\phi_1$ and $\phi_2$ have disjoint supports; no corner convention is needed.
Up to a rotation by $\pi/2$ this is configuration BC4 of Table~\ref{tab:bc}
with $c=0.25$.  We take $c=0.15$, $h=0.01$ and $\varepsilon=10^{-9}$; the value
$c=0.25$ is degenerate for this family and is treated in
Section~\ref{subsec:degenerate}.

Both initialisations of Algorithm~\ref{alg:monotone_iterative}, the $\varepsilon$-continuation from the
harmonic extension, and the iteration run directly at $\varepsilon=10^{-9}$
from a state in which $u_3$ already vanishes in the middle of the square ---
produce the same penalised solution, with $E_h=3.6328$ and
$P_h^\varepsilon=1.18\cdot10^{-2}$.  It exhibits the ternary structure: $u_3$
vanishes on an interior region of area $1.362$, on which $u_1$ and $u_2$
coexist, while outside it $u_1u_2=0$ and all three components may be positive.
With
\[
  E_h^{\mathrm{harm}}=3.0202,\qquad E_h(\text{penalised})=3.6328,\qquad
  E_h(\text{projected gradient})=3.6859,
\]
the bound \eqref{eq:harmonic_lower_bound} holds and the constraint costs about
$20\%$ of the harmonic energy.

Started from the projection of the penalised solution, the safeguarded inertial
projected gradient reproduces its active face at \emph{every} interior node,
and the computed interfaces coincide to plotting accuracy
(Figure~\ref{fig:surfaces}).  The remaining difference of $1.42\%$ is the
irreducible one between the two formulations: the penalised configuration
violates the constraint by $P_h^\varepsilon>0$ whereas the projected one
satisfies it exactly at the nodes, so $E_h^{\mathrm{pen}}\le E_h^{\mathrm{pg}}$
is forced and the gap is of the order of the penalty term.  Its fourth digit is
not significant, since where two components are exactly equal the projection may
assign either of them to the vanishing face.

The initialisation of the projected gradient is not innocuous, however.  At the
centre of the square the harmonic extensions are $(0.3376,0.3376,0.1500)$, so
$u_3$ is the smallest and the projection that initialises
Algorithm~\ref{alg:projected_gradient} makes it vanish there; the region so
produced is not the right one, Lemma~\ref{lem:facelock} freezes it, and the
iteration converges to a configuration with $E_h=3.7077$, higher by $0.6\%$ and
differing from the penalised solution on $8.6\%$ of the interior nodes.  We initialise the projected gradient from the penalised solution
throughout, at negligible cost once the latter has been computed.
\end{example}

\begin{figure}[htbp]
\centering
\includegraphics[width=0.94\textwidth]{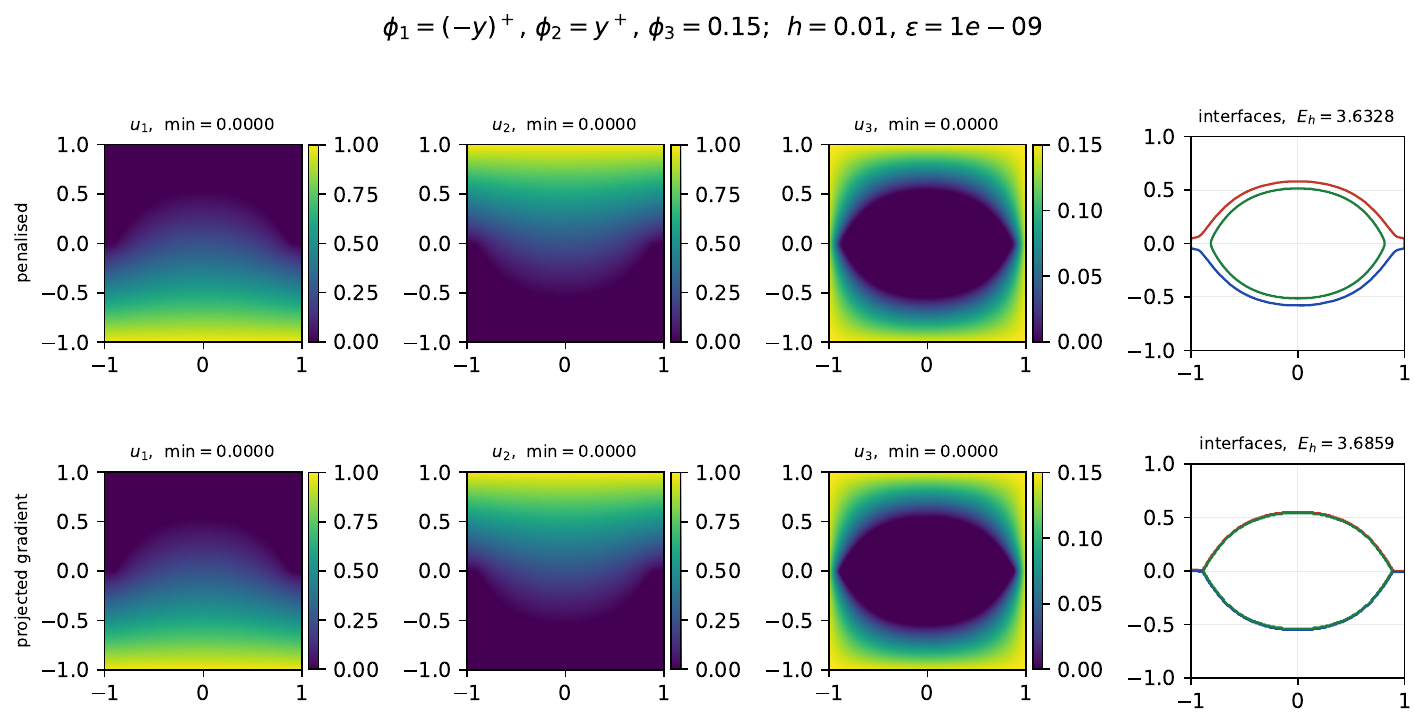}
\caption{Example~\ref{ex:numerical_square} at $c=0.15$, $h=0.01$,
$\varepsilon=10^{-9}$.  Top: the penalised solution.  Bottom: the projected
gradient started from its projection.  The two produce the same interfaces, and
$u_3$ vanishes on an interior region on which $u_1$ and $u_2$ coexist.}
\label{fig:surfaces}
\end{figure}

\begin{example}\label{ex:fista_square_9bc}
We apply both methods to the nine discrete boundary configurations of
Table~\ref{tab:bc}, adapted from \cite{Bidentical}.  Here $z=(x,y)$, $|\cdot|$
is the Euclidean norm and $\theta=\operatorname{atan2}(y,x)$.  On each
\emph{open} edge the components not listed are set to zero; at a node where two
prescriptions meet --- the four corners and the split node of BC8 --- the raw
vector is their componentwise maximum, projected onto $\Sigma$ by the rule of
Lemma~\ref{lem:projection_corrected} (at the corners of BC5, for instance,
$(1,1,0.3)\mapsto(1,1,0)$).  BC8 is written with $x\le0$ so that no isolated
node with $\phi=(0,0,0)$ arises.  Every configuration satisfies
$\phi_1\phi_2\phi_3=0$ on $\partial\Omega$: for BC1--2 because
$\sum_i\cos(\theta-2\pi i/3)\equiv0$, so the three cosines cannot be
simultaneously positive; for BC6 because on each edge one of $|z-c_i|\ge2$
identically.  BC3, BC5 and BC7--9 are discontinuous at corners or at an edge
split, hence discrete data rather than $H^{1/2}$ traces; the compatibility
hypothesis of Assumption~\ref{ass:boundary_data} is not asserted for them and
they are not used in the proofs of Section~\ref{sec: Penal_algorithms}.
 
\begin{table}[htbp]
\centering\small
\caption{Discrete boundary configurations on $\Omega=[-1,1]^2$; components not
listed on an open edge are zero, conflicting nodes as described in the text.}
\label{tab:bc}
\renewcommand{\arraystretch}{1.2}
\begin{tabular}{cl}
\hline
BC & Definition on the open edges of $\partial\Omega$ \\
\hline
1 & $\phi_i=\max\{0,\cos(\theta-2\pi i/3)\}$, $i=1,2,3$ \\
2 & $\phi_i=\max\{0,\cos(\theta-2\pi i/3-\pi/4)\}$, $i=1,2,3$ \\
3 & $\phi_1=1$ on $y=-1$; $\phi_2=1$ on $y=1$; $\phi_3=0.5$ on $x=\pm1$ \\
4 & $\phi_1=x^+$; $\phi_2=(-x)^+$; $\phi_3=0.25$ on $\partial\Omega$ \\
5 & $\phi_1=1$ on $y=\pm1$; $\phi_2=1$ on $x=\pm1$; $\phi_3=0.3$ on $\partial\Omega$ \\
6 & $\phi_i=\max\{0,1-|z-c_i|/2\}$; $c_1=(-1,-1)$, $c_2=(1,1)$, $c_3=(1,-1)$ \\
7 & $\phi_1=\sin\!\bigl(\pi(x+1)/2\bigr)$, $\phi_2=\bigl(\cos(\pi(x+1)/2)\bigr)^+$
    on $y=\pm1$; $\phi_3=0.3$ on $x=\pm1$ \\
8 & $\phi_1=1$ on $y=-1$, $x\le0$; $\phi_2=1$ on $y=-1$, $x>0$; $\phi_3=1$ on $y=1$ \\
9 & $\phi_1=1$ on $y=-1$ or $x=-1$; $\phi_2=1$ on $y=1$ or $x=1$;
    $\phi_3=0.2$ on $\partial\Omega$ \\
\hline
\end{tabular}
\end{table}
 
Table~\ref{tab:energies} collects the results on the grid $h=0.01$;
Figures~\ref{fig:4} and~\ref{fig:5} show the threshold interfaces on the grid
$h=0.005$, whose data are in Table~\ref{tab:fine}.  Five observations.
 
\emph{(a) Ordering.}  $E_h^{\mathrm{harm}}\le E_h^{\mathrm{pen}}\le
E_h^{\mathrm{pg}}$ in every case, the middle inequality for the reason given in
Example~\ref{ex:numerical_square}.
 
\emph{(b) The default initialisations may select different branches.}  The gap
$E_h^{\mathrm{pg}}/E_h^{\mathrm{pen}}-1$ obtained from the harmonic start is
below $1\%$ for five configurations but reaches $2.88\%$ for BC5 and $4.70\%$
for BC4.  The accompanying column, the percentage of interior nodes at which
the two solutions vanish in different components, identifies the cause: where
it is a fraction of a percent the two outputs are the same configuration and
differ only by interface position, and where it exceeds $7\%$ they are
different configurations, as in Section~\ref{subsec:degenerate}.  These
columns are reported precisely to exhibit that effect, and not as a measure of
the accuracy of either method.
 
\emph{(c) On a common branch the methods agree everywhere.}  The last two
columns repeat the projected-gradient computation from the pointwise projection
of the penalised solution.  The active faces then agree at \emph{every}
interior node of \emph{every} configuration, and the gap collapses to the
irreducible one: $4.70\%\to0.76\%$ on BC4, $2.88\%\to0.24\%$ on BC5,
$1.15\%\to0.22\%$ on BC9, mean $1.78\%\to0.54\%$.  It is not uniformly better
--- on BC1 it gives $4.2201$ against $4.2181$ --- as must be expected from a
method without a global optimality guarantee.  We nevertheless recommend
initialising the projected gradient from the penalised solution, at negligible
extra cost once the latter is available.
 
\emph{(d) Several discrete fixed points.}  For BC5 the unaccelerated and the
Anderson-accelerated iteration, started from the same state at
$\varepsilon=10^{-7}$, reach two distinct fixed points of the same map at
$\varepsilon=10^{-9}$, with energies $47.5216$ and $47.5340$, both verified to
a residual below $10^{-11}$, differing by up to $7\cdot10^{-2}$ at $1336$ nodes
while agreeing on the active face at $99.96\%$ of the interior nodes; the mesh
continuation of Table~\ref{tab:fine} reaches a third, $47.5619$.  BC9 behaves
alike, with $22.0446$ and $22.1179$.  These are exactly the two configurations
with four unit corner jumps, no other one shows the effect, and at $h=0.02$ the
BC5 paths agree to $2\cdot10^{-4}$, so unlike the one-dimensional branches of
Example~\ref{ex:twofixedpoints} the multiplicity is tied to the corner
singularities being resolved.  A residual-based stopping test, however tight,
cannot say which solution has been computed; the bracketing of
Theorem~\ref{thm:sandwich}, which bounds every fixed point from a single
undamped run, is the tool that addresses this.
 
\emph{(e) Mesh dependence.}  Table~\ref{tab:fine} repeats all configurations on
$h=0.005$ by continuation from $h=0.02$, with the projected gradient
initialised from the penalised solution.  For the continuous data (BC1, BC2,
BC4, BC6) $E_h^{\mathrm{pen}}$ decreases by $0.23$--$0.28\%$; for the jump data
it \emph{increases} by $8$--$12\%$.  The increase is the logarithmic divergence
of the harmonic energy at a jump: near a boundary point of interior angle
$\alpha$ where the data jump by $J$, the harmonic extension behaves like
$(J/\alpha)\theta$, with energy $(J^2/\alpha)\ln(R/h)$ in the annulus $h<r<R$,
so every halving of $h$ adds
\begin{equation}\label{eq:logdiv}
  \Delta E_h^{\mathrm{harm}}=\frac{\ln2}{\pi}\Bigl(2\!\!\sum_{\text{corners}}\!J^2
  +\sum_{\text{edges}}\!J^2\Bigr).
\end{equation}
For BC3, with corner jumps $1,1,1,1,\frac12,\frac12,\frac12,\frac12$, this
predicts $2.2064$ per halving, and the computed $E_h^{\mathrm{harm}}$ at
$h=0.02,0.01,0.005$ are $19.1900$, $21.3964$, $23.6028$, increments $2.2064$
and $2.2064$; predicted and computed increments agree to four decimals for BC5
($3.5302$, $3.5303$), BC8 and BC9 as well.  The
divergence lives entirely in the harmonic part: the \emph{segregation cost}
$E_h^{\mathrm{pen}}-E_h^{\mathrm{harm}}$ is mesh-convergent even for jump data
($0.05\%$ for BC3, $0.00\%$ for BC8, $0.6\%$ for BC7), and is the quantity to
compare across grids for such data.  The exceptions, BC5 and BC9 with $16\%$,
are the two configurations of (d), where the continuation path lands on a
different fixed point.
 
Finally, Table~\ref{tab:fine} (right) isolates the interface error for BC1,
whose data are continuous, by continuation to $h=0.0025$, comparing consecutive
grids through the symmetric Hausdorff distance between the $\delta$-level sets.
Both converge at first order: the energy differences halve
($1.86\cdot10^{-2}$, $9.46\cdot10^{-3}$, $4.76\cdot10^{-3}$, Richardson limit
$4.1669$) and the interface displacement is $0.68h$, $0.60h$, $0.60h$, equal
for the three components as the $2\pi/3$ symmetry of BC1 requires.  On the
finest grid, however, changing $\delta$ from $10^{-5}$ to $10^{-4}$ moves the
contours by $3.4\cdot10^{-2}$, ten times the mesh error at that level, and by
the same amount to three digits at $h=0.005$.  Below $h\approx0.005$ the
threshold convention, not the mesh, limits interface accuracy, because the
layer width $\sqrt{\varepsilon/q_0}\approx3\cdot10^{-5}$ remains unresolved.
Refinement does not address branch selection either.
 
\begin{table}[htbp]
\centering\small
\caption{Nine configurations, $h=0.01$, $\varepsilon=10^{-9}$.  ``gap'' is
$E_h^{\mathrm{pg}}/E_h^{\mathrm{pen}}-1$ and ``faces'' the percentage of
interior nodes at which the two solutions vanish in different components.  The
last two columns start the projected gradient from the projection of the
penalised solution; the active faces then agree at every node.}
\label{tab:energies}
\renewcommand{\arraystretch}{1.1}
\begin{tabular}{crrrrrrr}
\hline
 & & & \multicolumn{3}{c}{harmonic start} & \multicolumn{2}{c}{penalised start}\\
BC & $E_h^{\mathrm{harm}}$ & $E_h^{\mathrm{pen}}$ & $E_h^{\mathrm{pg}}$ & gap & faces
   & $E_h^{\mathrm{pg}}$ & gap \\
\hline
1 &  3.4479 &  4.1859 &  4.2181 & $0.77\%$ & $0.18\%$ &  4.2201 & $0.82\%$\\
2 &  3.4169 &  4.1575 &  4.1904 & $0.79\%$ & $0.24\%$ &  4.1902 & $0.79\%$\\
3 & 21.3964 & 21.8171 & 21.8311 & $0.06\%$ & $0.05\%$ & 21.8303 & $0.06\%$\\
4 &  3.0202 &  3.9718 &  4.1586 & $4.70\%$ & $8.12\%$ &  4.0021 & $0.76\%$\\
5 & 34.2482 & 47.5340 & 48.9008 & $2.88\%$ & $12.88\%$ & 47.6464 & $0.24\%$\\
6 &  2.0150 &  2.5771 &  2.6259 & $1.89\%$ & $7.18\%$ &  2.6065 & $1.14\%$\\
7 & 12.2304 & 13.2447 & 13.3471 & $0.77\%$ & $5.01\%$ & 13.3211 & $0.58\%$\\
8 & 21.9330 & 22.1270 & 22.1489 & $0.10\%$ & $3.68\%$ & 22.1464 & $0.09\%$\\
9 & 17.8067 & 22.0446 & 22.2975 & $1.15\%$ & $7.06\%$ & 22.0938 & $0.22\%$\\
\hline
\end{tabular}
\end{table}
 
\begin{table}[htbp]
\centering\small
\caption{Left: the nine configurations at $h=0.005$ by continuation from
$h=0.02$, the projected gradient started from the penalised solution; ``cost''
is $E_h^{\mathrm{pen}}-E_h^{\mathrm{harm}}$ on the two grids.  Right: mesh
convergence for BC1 at fixed $\varepsilon=10^{-9}$, where $\Gamma_i$ is the
Hausdorff distance between the $\delta$-level sets on consecutive grids.}
\label{tab:fine}
\renewcommand{\arraystretch}{1.1}
\begin{tabular}{crrrrrr}
\hline
BC & $E_h^{\mathrm{harm}}$ & $E_h^{\mathrm{pen}}$ & $E_h^{\mathrm{pg}}$ & gap
   & cost $.01$ & cost $.005$\\
\hline
1 &  3.4382 &  4.1764 &  4.2140 & $0.90\%$ &  0.7380 &  0.7382\\
2 &  3.4072 &  4.1481 &  4.1859 & $0.91\%$ &  0.7406 &  0.7409\\
3 & 23.6028 & 24.0237 & 24.0403 & $0.07\%$ &  0.4207 &  0.4209\\
4 &  3.0101 &  3.9624 &  4.0002 & $0.95\%$ &  0.9516 &  0.9523\\
5 & 37.7785 & 53.2531 & 53.5848 & $0.62\%$ & 13.2858 & 15.4746\\
6 &  2.0075 &  2.5698 &  2.6030 & $1.29\%$ &  0.5621 &  0.5623\\
7 & 13.2530 & 14.2735 & 14.3741 & $0.70\%$ &  1.0143 &  1.0205\\
8 & 24.1394 & 24.3334 & 24.3544 & $0.09\%$ &  0.1940 &  0.1940\\
9 & 19.5718 & 24.5075 & 24.6449 & $0.56\%$ &  4.2379 &  4.9357\\
\hline
\end{tabular}\qquad
\begin{tabular}{crrrr}
\hline
$h$ & $E_h$ & $\Gamma_1$ & $\Gamma_2$ & $\Gamma_3$\\
\hline
$0.02$   & 4.20451 & 0.0135 & 0.0135 & 0.0129\\
$0.01$   & 4.18588 & 0.0060 & 0.0060 & 0.0056\\
$0.005$  & 4.17642 & 0.0030 & 0.0030 & 0.0028\\
$0.0025$ & 4.17166 & & & \\
\hline
\end{tabular}
\end{table}
 
\begin{figure}[htbp]
\centering
\includegraphics[width=0.95\textwidth]{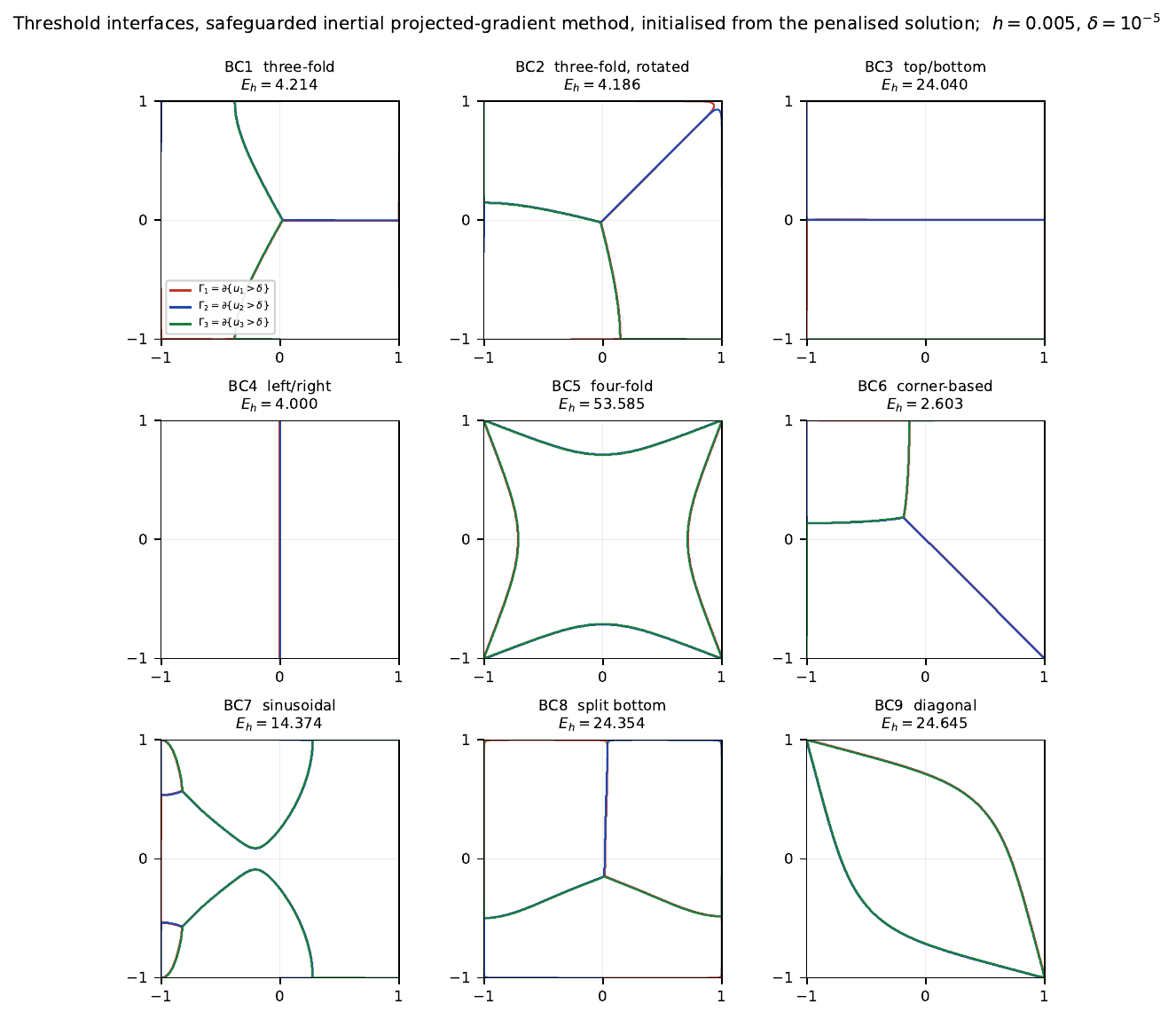}
\caption{Threshold interfaces \eqref{eq:threshold} of the safeguarded inertial
projected-gradient method with the exact projection, initialised from the
penalised solution; $h=0.005$, $\delta=10^{-5}$, energies $E_h$ of
\eqref{eq:discrete_energy}.}
\label{fig:4}
\end{figure}
 
\begin{figure}[htbp]
\centering
\includegraphics[width=0.8\textwidth]{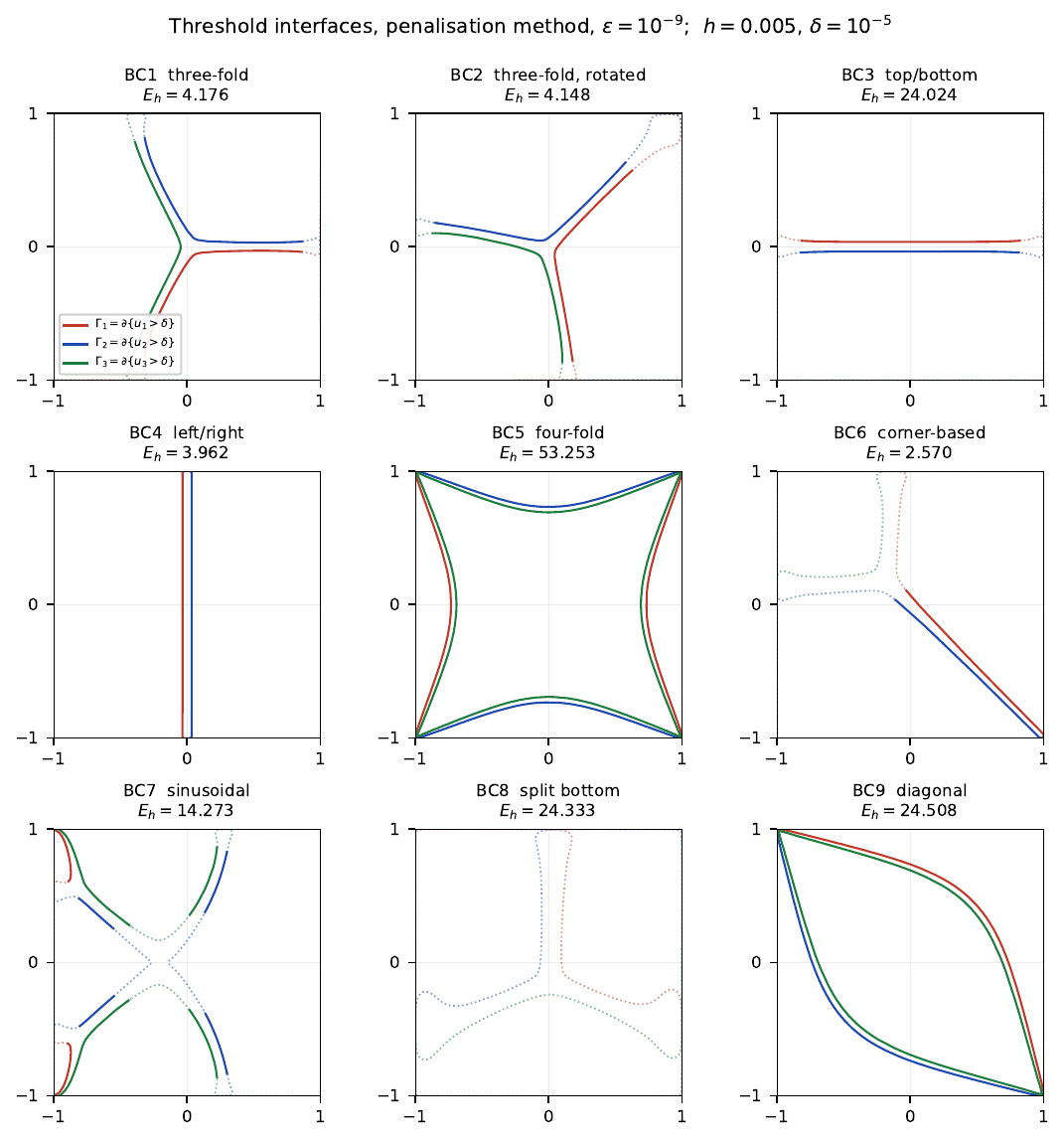}
\caption{Threshold interfaces \eqref{eq:threshold} of the penalised solutions,
$h=0.005$, $\varepsilon=10^{-9}$, $\delta=10^{-5}$.  Solid: the resolved part
\eqref{eq:resolved}.  Dotted: the part where the competing coefficient
degenerates and the level set is an artifact of $\delta$.  A component whose
threshold set is all of $\Omega_h$ contributes no curve.}
\label{fig:5}
\end{figure}
\end{example}
 
\subsection{Degenerate and non-convergent instances}\label{subsec:degenerate}
 
The analysis distinguishes four failure mechanisms, each illustrated below by
one computation: (i) the fixed-point iteration may fail to converge even when
the fixed point exists and is unique, as permitted by Theorem~\ref{thm:sandwich}
(Example~\ref{ex:twocycle}); (ii) when $L_T\ge1$ the penalised equation may
possess several fixed points, selected by the initialisation
(Examples~\ref{ex:numerical_square} and~\ref{ex:twofixedpoints}); (iii) a
descent method
may converge to a non-minimising stationary configuration, because the
projection onto a union of faces cannot move an interface across a region where
two components are positive (Lemma~\ref{lem:facelock},
Example~\ref{ex:facelock}); (iv) the limit problem may have several global
minimisers (Example~\ref{ex:nonuniqueness}).  Unless a proof is given,
conclusions drawn from these finite-dimensional computations concern only the
corresponding discretisations.  Mechanism (ii) already occurs in the family of
Example~\ref{ex:numerical_square}, at the boundary value $c=0.25$ used for BC4.
There the penalised system \eqref{eq:system} has two solutions at
$\varepsilon=10^{-9}$, both fixed points to a residual below $10^{-11}$: the
$\varepsilon$-continuation from the harmonic extension gives $E_h=3.9718$, with
$u_3\ge0.2435$ throughout and $u_1,u_2$ pairwise segregated along
$\{y\approx0\}$, so that the ternary constraint is met pairwise and
$\Gamma_{3,h}^{\varepsilon,\delta}$ is empty; the iteration run directly at
$\varepsilon=10^{-9}$ from a state with $u_3=0$ in the middle gives
$E_h=4.0683$, with $u_3$ vanishing on a lens.  Started from the projection of
either, the projected gradient reproduces that solution's active face at every
interior node, reaching $4.0035$ and $4.1291$; the methods agree on each
branch, the branches do not agree with each other.
 
Table~\ref{tab:lens} explains why $c=0.25$ is delicate.  Within the family in
which $u_3$ vanishes on a disc $B_r$ centred at the origin, $u_1$ and $u_2$
coexisting there and pairwise segregated across $\{y=0\}$ outside it, each
member is a constrained minimiser obtained from three masked Dirichlet
problems, and $r=0$ is the pairwise configuration.  Since $u_3$ solves a linear
problem with data proportional to $c$, the energy is $e_{12}(r)+c^2\hat e_3(r)$
with $c$-independent coefficients, so one scan settles every $c$: enlarging the
disc always lowers $e_{12}$ and always raises $\hat e_3$, and a small lens never
pays, because forcing $u_3$ to vanish on a disc costs an amount governed by its
capacity, which decreases only logarithmically as $r\downarrow0$, while the gain
is $O(r^2)$.  A disc bounds the true minimum only from above, so the transition
is better located with the solvers, which are free to choose the shape: the area
of $\{u_3=0\}$ in the minimiser decreases continuously from $1.43$ at $c=0.15$
to $0.665$ at $c=0.247$, then drops to zero between $c=0.247$ and $c=0.248$.
The value $c=0.25$ therefore lies about one percent above a transition, which is
what makes it so sensitive to the initialisation, the solver and the arithmetic,
and is why Example~\ref{ex:numerical_square} uses $c=0.15$ instead.
 
\begin{table}[htbp]
\centering\small
\caption{Example~\ref{ex:numerical_square}.  Left: constrained minimiser within
the family ``$u_3=0$ on the disc $B_r$'' at $h=0.01$; the energy for boundary
value $c$ is $e_{12}(r)+c^2\hat e_3(r)$, and the minimising row is in bold.
Right: both branches computed with the solvers at $h=0.02$, which are not
restricted to discs; ``area'' is $\vert{}\{u_3<10^{-3}\}\vert{}$ in the minimiser.}
\label{tab:lens}
\renewcommand{\arraystretch}{1.1}
\begin{tabular}{crrrr}
\hline
$r$ & $e_{12}(r)$ & $\hat e_3(r)$ & $c=0.15$ & $c=0.25$\\
\hline
$0$    & 4.0200 & 0.000 & 4.0200 & \textbf{4.0200}\\
$0.10$ & 4.0080 & 2.598 & 4.0664 & 4.1703\\
$0.30$ & 3.9155 & 4.858 & 4.0248 & 4.2191\\
$0.50$ & 3.7366 & 8.086 & 3.9185 & 4.2420\\
$0.70$ & 3.4866 & 14.364 & \textbf{3.8098} & 4.3844\\
\hline
\end{tabular}\qquad
\begin{tabular}{crrr}
\hline
$c$ & branch 1 & branch 2 & area\\
\hline
$0.150$ & 3.6912 & 3.6912 & 1.432\\
$0.240$ & 4.0919 & 4.0901 & 0.700\\
$0.246$ & 4.1157 & 4.1158 & 0.668\\
$0.247$ & 4.1196 & 4.1196 & 0.665\\
$0.248$ & 4.0028 & 4.1224 & $0$\\
$0.250$ & 4.0028 & 4.1316 & $0$\\
\hline
\end{tabular}
\end{table}
 
Examples~\ref{ex:twofixedpoints}
and~\ref{ex:facelock} use the data of Example~\ref{ex:nonuniqueness},
\begin{equation}\label{eq:data24}
\Omega=(0,1),\qquad U(0)=(2,1,0),\qquad U(1)=(0,1,3),
\end{equation}
for which the constrained minimum $25$ is attained by $U^{(A)}$ and $U^{(B)}$.
 
\begin{example}[Two-cycle of the undamped Picard iteration with a unique fixed
point]\label{ex:twocycle}
With a single interior node ($h=1/2$) the update \eqref{eq:interior-system}
reads $v_i=a_i/(1+c\,(u_ju_\ell)^2)$ with $c=h^2/(2\varepsilon)$ and $a_i$ the
harmonic-extension value.  For the symmetric data $\varphi(0)=(2a,2a,0)$,
$\varphi(1)=(0,0,2a)$ one has $a_i=a$, the iterates stay symmetric, and
$u_{k+1}=f(u_k)$ with $f(u)=a/(1+cu^4)$.  Since $f$ is decreasing the fixed
point $u^*$ is unique, while $f'(u^*)=-4s/(1+s)$ with $s:=c(u^*)^4$, so $u^*$
is locally repelling if and only if $s>1/3$.  For $a=1$ and $c=100$, that is
$\varepsilon=1/800$, one finds $u^*=0.363692$, $s=1.749581$ and
$f'(u^*)=-2.5452$; from the harmonic value $u^0=1$ the odd and even iterates
converge monotonically, as in Theorem~\ref{thm:sandwich}(ii), to the two-cycle
$\underline u=0.009901$, $\bar u=0.999999$.  The relaxed map has
$g_\omega'(u^*)=1-\omega\bigl(1+\frac{4s}{1+s}\bigr)$ and converges locally if
and only if $\omega<2(1+s)/(1+5s)=0.5641$; $\omega=0.5$ returns to $u^*$.
 
The same occurs on a fine grid.  For \eqref{eq:data24} with $N=200$ the
undamped iteration stalls at the successive differences of
Table~\ref{tab:twocycle} while $\omega=0.5$ converges, and
$L_T=4M^4/(\varepsilon\lambda_\Omega)=324/(\pi^2\varepsilon)\approx33/\varepsilon\gg1$.
Note that \eqref{eq:bracketing_gap} bounds the gap $\bar u-\underline u$
\emph{from above} by the successive difference: it guarantees that the stopping
test cannot fire while the gap is large, not that a large difference proves
divergence.
\end{example}
 
\begin{table}[htbp]
\centering\small
\caption{Picard iteration for \eqref{eq:data24}, $N=200$: stalling value of
$\Vert{}U^{k+1}-U^k\Vert{}$ for $\omega=1$ in the two norms, and iterations to
$\Vert{}U^{k+1}-U^k\Vert{}_X\le10^{-10}$ for $\omega=0.5$.}
\label{tab:twocycle}
\renewcommand{\arraystretch}{1.1}
\begin{tabular}{crrrr}
\hline
$\varepsilon$ & $\Vert{}\cdot\Vert{}_X$ & $\Vert{}\cdot\Vert{}_\infty$ & its.\ ($\omega=0.5$) & $L_T$\\
\hline
$10^{-2}$ & 0.8093 & 1.2304 & 182 & $3.3\cdot10^{3}$\\
$10^{-3}$ & 1.0700 & 1.7215 & 279 & $3.3\cdot10^{4}$\\
$10^{-4}$ & 1.2461 & 2.0865 & 391 & $3.3\cdot10^{5}$\\
$10^{-5}$ & 1.3722 & 2.3517 & 629 & $3.3\cdot10^{6}$\\
\hline
\end{tabular}
\end{table}
 
\begin{example}[Initialization-dependent discrete fixed points at fixed
$\varepsilon$]\label{ex:twofixedpoints}
Take the discretisation of \eqref{eq:system} with the data \eqref{eq:data24},
$\omega=0.5$ and the stopping test $\Vert{}U^{k+1}-U^k\Vert{}_X\le10^{-10}$.  At $N=200$
and $\varepsilon=10^{-4}$ the harmonic initialisation and the discrete
$U^{(B)}$ converge to the same fixed point $\widehat U_h$, for which
$\widehat u_{2,h}<10^{-3}$ exactly on the nodes of $(0.370,0.540)$ and
$E_h^\varepsilon(\widehat U_h)=23.1773$, whereas the discrete $U^{(A)}$
converges to a distinct fixed point $\widetilde U_h$ with
$\min\widetilde u_{2,h}=0.8797$ and
$E_h^\varepsilon(\widetilde U_h)=23.0961$; see Figure~\ref{fig:1d}(a,b).  The
harmonic start selects the \emph{larger} energy, as
Remark~\ref{rem:minimizer_vs_fixedpoint} permits.  Both values lie below
$m_0=25$, as $E^\varepsilon(U^\varepsilon)\le E(U^*)=m_0$ requires, and rise
towards $m_0$ as $\varepsilon\downarrow0$ (Table~\ref{tab:branches}), in
accordance with (9).
 
Table~\ref{tab:branches} shows that the branches are not a discretisation
artifact: refining $N=200\to800$ changes each energy by less than
$3\cdot10^{-4}$, while their separation, $0.08$ at $\varepsilon=10^{-4}$, is
four orders of magnitude larger, and the harmonic start selects the higher
branch at every $(N,\varepsilon)$.  This is numerical evidence for
non-uniqueness of solutions of the \emph{penalised} system \eqref{eq:system} at
fixed $\varepsilon$, of which Example~\ref{ex:numerical_square} exhibits a
two-dimensional instance, and is a different statement from the proven
non-uniqueness of the limiting minimiser in Example~\ref{ex:nonuniqueness}.  It
remains an assertion about the discretisations described.
\end{example}
 
\begin{table}[htbp]
\centering\small
\caption{$E_h^\varepsilon$ of the two branches of
Example~\ref{ex:twofixedpoints} under refinement ($\omega=0.5$, stopping test
$\Vert{}U^{k+1}-U^k\Vert{}_X\le10^{-10}$).}
\label{tab:branches}
\renewcommand{\arraystretch}{1.1}
\begin{tabular}{crrrrrr}
\hline
 & \multicolumn{2}{c}{$\varepsilon=10^{-3}$} & \multicolumn{2}{c}{$\varepsilon=10^{-4}$}
 & \multicolumn{2}{c}{$\varepsilon=10^{-5}$}\\
$N$ & harmonic & $U^{(A)}$ & harmonic & $U^{(A)}$ & harmonic & $U^{(A)}$\\
\hline
200 & 21.8964 & 21.6650 & 23.1773 & 23.0961 & 23.9475 & 23.9206\\
400 & 21.8967 & 21.6652 & 23.1779 & 23.0966 & 23.9486 & 23.9217\\
800 & 21.8968 & 21.6653 & 23.1780 & 23.0967 & 23.9488 & 23.9219\\
\hline
\end{tabular}
\end{table}
 
\begin{lemma}[Face locking]\label{lem:facelock}
Let $z\in\mathfrak C_h$, let $r$ be an interior node with $z_k(r)=0$ and
$\min_{m\ne k}z_m(r)>2\tau\max_m\vert{}(\nabla_mF_h(z))(r)\vert{}$.  Then every
$z^+\in\operatorname{Proj}_{\mathfrak C_h}(z-\tau\nabla F_h(z))$ satisfies
$z^+_k(r)=0$: the active face at $r$ is unchanged.
\end{lemma}
 
\begin{proof}
Let $\tilde{z}(r) := z(r) - \tau (\nabla F_h(z))(r)$ denote the unconstrained trial step at the interior node $r$. Let $g_m := (\nabla_mF_h(z))(r)$ denote the local gradient components, and define by
$$\eta := \min_{m\ne k}z_m(r).$$
By the lemma's hypothesis we have $\eta > 2\tau\max_m\vert{}g_m\vert{}$.

We evaluate the positive part of the unconstrained step, $\tilde{z}^+(r) := \max(0, \tilde{z}(r))$, component by component. 

First, consider the $k$-th component. Since the current iterate lies on the face where $z_k(r) = 0$, the unconstrained update is simply $\tilde{z}_k(r) = -\tau g_k$. Its positive part is strictly bounded by 
\begin{equation} \label{eq:facelock_k}
\tilde{z}^+_k(r) = \max(0, -\tau g_k) \le \tau |g_k| \le \tau \max_m |g_m| < \frac{\eta}{2}.
\end{equation}

Next, consider the remaining components $m \neq k$. By definition of $\eta$, we have $z_m(r) \ge \eta$. The unconstrained update for these components is bounded from below as follows
\begin{equation} \label{eq:facelock_m}
\tilde{z}_m(r) = z_m(r) - \tau g_m \ge \eta - \tau|g_m| \ge \eta - \tau \max_m |g_m|.
\end{equation}
Because $\tau \max_m \vert{}g_m\vert{} < \eta/2$, equation \eqref{eq:facelock_m} implies that $\tilde{z}_m(r) > \eta - \eta/2 = \eta/2$. Since this value is strictly positive, we obtain
\begin{equation} \label{eq:facelock_m_pos}
\tilde{z}^+_m(r) = \tilde{z}_m(r) > \frac{\eta}{2} \qquad \text{for all } m \neq k.
\end{equation}
Comparing \eqref{eq:facelock_k} and \eqref{eq:facelock_m_pos}, we get
\[\tilde{z}^+_k(r) < \frac{\eta}{2} < \tilde{z}^+_m(r) \qquad \text{for all } m \neq k.\]
This strict inequality guarantees that $\tilde{z}^+_k(r)$ is the unique minimum among the components of $\tilde{z}^+(r)$.

By the pointwise projection characterization in Lemma \ref{lem:projection_corrected}, the exact metric projection $\operatorname{Proj}_{\mathfrak C_h}$ evaluates the positive part of the trial vector and strictly enforces the segregation constraint by setting the smallest component(s) to zero. Because $\tilde{z}^+_k(r)$ is strictly the minimal component, the projection yields $z^+_k(r)=0$.
\end{proof}
 
Lemma~\ref{lem:facelock} is a one-step condition; if it holds at every iterate,
which can be monitored on the computed sequence, the vanishing component at $r$
never changes.  On a set of nodes with a fixed active face the
projected-gradient iteration reduces to gradient descent for $E_h$ restricted
to that face pattern, whose minimiser need not be a global minimiser:
Algorithm~\ref{alg:projected_gradient} can switch the vanishing component only
where two components are of size $O(\tau\Vert{}\nabla F_h\Vert{}_\infty)$.
 
\begin{example}[Projected gradient converging to a face-locked
non-minimizer]\label{ex:facelock}
Take \eqref{eq:data24} with $N=200$ and $\tau=0.2h^2$, so that
$\tau\vert{}\lambda_{\max}(A)\vert{}=0.80$.  Projecting the harmonic extension
$(2-2x,1,3x)$ gives $u_3=0$ on $(0,\frac13)$, $u_2=0$ on $(\frac13,\frac12)$
and $u_1=0$ on $(\frac12,1)$; on the grid the vanishing sets are the nodes of
$[0,0.330]$, $[0.335,0.495]$ and $[0.500,1]$.  This nodal pattern is unchanged
at every iteration of the run, and the limit is exactly the minimiser of $E_h$
over the frozen pattern,
\[E_h=\frac{4}{0.5}+\frac{1}{0.335}+\frac{1}{0.505}+\frac{9}{0.67}=26.39811 ,\]
against the computed $26.3981$.  The continuum pattern with
interfaces at $\frac13$ and $\frac12$ has energy $8+13.5+3+2=26.5$; the
difference is a grid effect, since $\frac13$ is not a node and the vanishing
sets of adjacent components are offset by one node.  Either way the limit is
bounded away from $25$ and is neither $U^{(A)}$ nor $U^{(B)}$.  Lowering $u_2$
by $1.2\,e^{-((x-0.5)/0.12)^2}$ before the initial projection gives the pattern
$u_2=0$ on $[0.305,0.640]$, again frozen, with
$E_h=4/0.645+1/0.305+1/0.360+9/0.700=25.11516$ against the computed $25.1152$,
close to $U^{(B)}$; see Figure~\ref{fig:1d}(c,d).  The safeguarded inertial
variant of Algorithm~\ref{alg:fista}, with the exact projection, $\tau_0=0.2h^2$ and
$\rho=0.5$, reaches the same two patterns and energies in $3719$ and $3990$
iterations against $183352$ and $198543$ for
Algorithm~\ref{alg:projected_gradient}: the acceleration is substantial but
does not change the selected pattern.
\end{example}
 
\begin{figure}[htbp]
\centering
\includegraphics[width=0.92\textwidth]{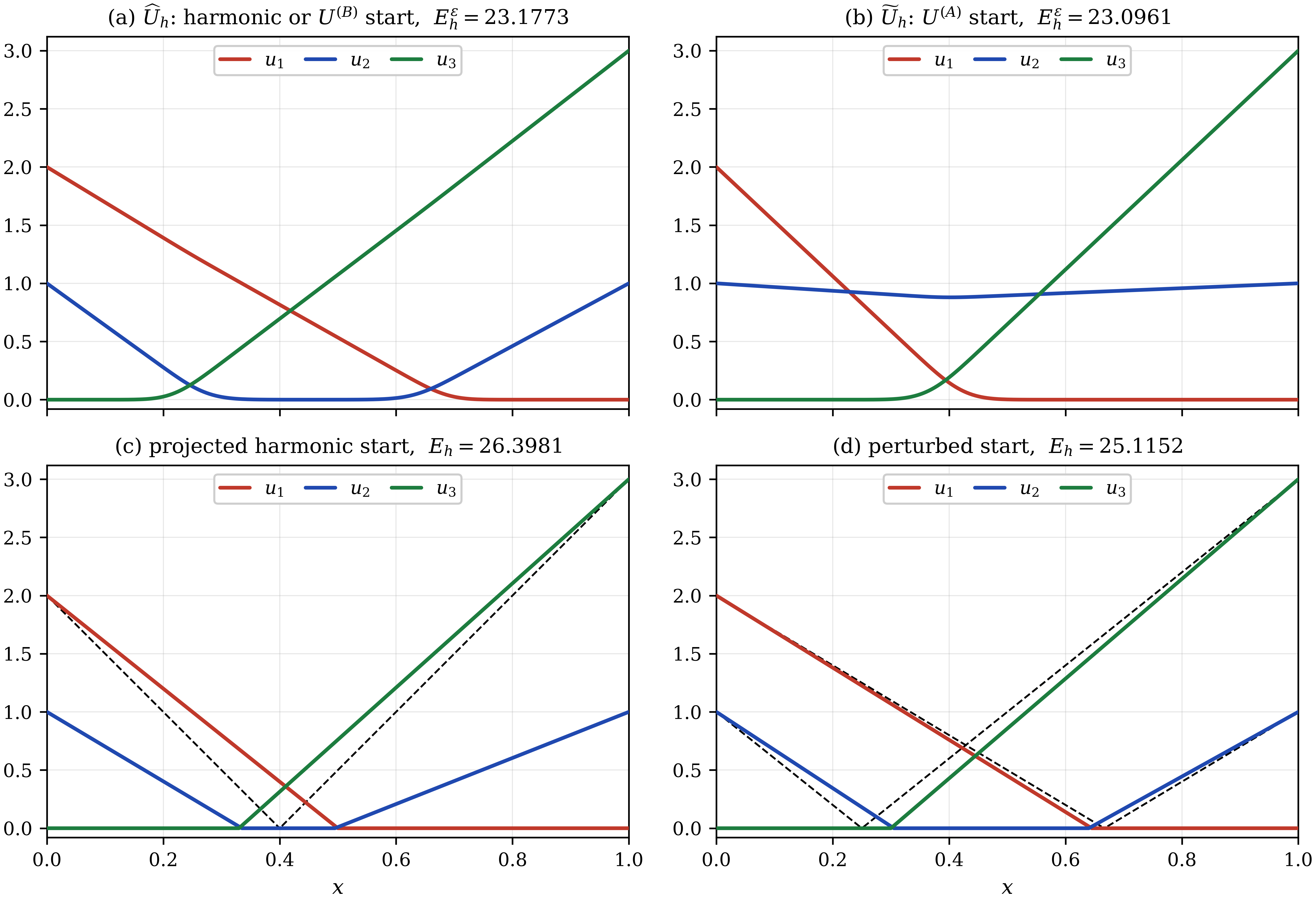}
\caption{(a,b) The two discrete fixed points of Example~\ref{ex:twofixedpoints}
at $N=200$, $\varepsilon=10^{-4}$, $\omega=0.5$.  (c,d) The face-locked limits
of Example~\ref{ex:facelock} at $N=200$, $\tau=0.2h^2$; dashed, the components
of $U^{(A)}$ and $U^{(B)}$, both of energy $25$.}
\label{fig:1d}
\end{figure}
 
\begin{remark}
Example~\ref{ex:twocycle} is a property of the iteration: the fixed point is
unique and damping cures the failure.  Example~\ref{ex:twofixedpoints}, observations (b) and (d) of
Example~\ref{ex:fista_square_9bc} and the family of
Example~\ref{ex:numerical_square} at $c=0.25$ are properties of the penalised
equation:
the branch is selected by the initialisation, by the continuation path or even
by the acceleration, and damping does not cure it.
Example~\ref{ex:facelock}, and the cold-started run of
Example~\ref{ex:numerical_square}, are properties of the projection onto the
union of faces, whose interfaces cannot cross regions where two components are
positive; initialising from the penalised solution cures this one.
Example~\ref{ex:nonuniqueness} is a property of the limit problem itself.  None
of these can be excluded a priori, which is why continuation, damping and
initialisation are treated throughout as branch-selection heuristics rather
than convergence guarantees, and why the two algorithms are compared only on a
common branch.
\end{remark}
 
\section*{Acknowledgments}
FB and VK received funding support from the Czech National Science Foundation under Project 24-11664S. 
\bibliographystyle{plain}
\bibliography{Refs}

\end{document}